\documentclass[reqno]{amsart}
\usepackage{amsmath, amssymb, amsthm}
\usepackage[margin=2.75cm]{geometry}
\usepackage[usenames]{color}
\usepackage{hyperref}

\newtheorem{thm}{Theorem}[section]
\newtheorem{cor}[thm]{Corollary}
\newtheorem{lem}[thm]{Lemma}
\newtheorem{prop}[thm]{Proposition}

\theoremstyle{definition}
\newtheorem{rmk}[thm]{Remark}

\numberwithin{equation}{section}
\allowdisplaybreaks
\def\bysame{\leavevmode\hbox to3em{\hrulefill}\thinspace}

\newcommand{\ep}{\epsilon}
\newcommand{\vep}{\varepsilon}
\newcommand{\pa}{\partial}

\newcommand{\N}{\mathbb{N}}
\newcommand{\R}{\mathbb{R}}
\newcommand{\ms}{\mathbb{S}}
\newcommand{\mca}{\mathcal{A}}
\newcommand{\mcc}{\mathcal{C}}
\newcommand{\mcd}{\mathcal{D}}
\newcommand{\mch}{\mathcal{H}}
\newcommand{\mci}{\mathcal{I}}
\newcommand{\mcl}{\mathcal{L}}
\newcommand{\mco}{\mathcal{O}}
\newcommand{\mcu}{\mathcal{U}}
\newcommand{\mcv}{\mathcal{V}}

\newcommand{\tu}{\tilde{u}}
\newcommand{\tv}{\tilde{v}}
\newcommand{\tw}{\tilde{w}}
\newcommand{\whf}{\widehat{F}}
\newcommand{\whg}{\widehat{G}}
\newcommand{\wtg}{\widetilde{G}}

\newcommand{\wtu}{\widetilde{U}}
\newcommand{\wtv}{\widetilde{V}}
\newcommand{\wtmcu}{\widetilde{\mcu}}
\newcommand{\wtmcv}{\widetilde{\mcv}}

\renewcommand{\(}{\left(}
\renewcommand{\)}{\right)}

\begin{document}

\title[Minimal energy solutions to the fractional Lane-Emden system, I]{Minimal energy solutions to the fractional Lane-Emden \\
system, I: Existence and singularity formation}

\author{Woocheol Choi}
\address[Woocheol Choi]{School of Mathematics, Korea Institute for Advanced Study, Seoul 130-722, Republic of Korea}
\email{wchoi@kias.re.kr}

\author{Seunghyeok Kim}
\address[Seunghyeok Kim]{Departamento de Matem\'{a}tica, Pontificia Universidad Cat\'{o}lica de Chile, Avenida Vicu\~{n}a Mackenna 4860, Santiago, Chile}
\email{shkim0401@gmail.com}

\thanks{W. Choi is partially supported by POSCO TJ Park Foundation in Republic of Korea. S. Kim is supported by FONDECYT Grant 3140530 in Chile.}

\begin{abstract}
This is the first of two papers which study asymptotic behavior of minimal energy solutions to the fractional Lane-Emden system in a smooth bounded domain $\Omega$
\[(-\Delta)^s u = v^p, \quad (-\Delta)^s v = u^q \text{ in } \Omega \quad \text{and} \quad u = v = 0 \text{ on } \pa \Omega \quad \text{for } 0 < s < 1\]
under the assumption that the subcritical pair $(p,q)$ approaches to the critical Sobolev hyperbola.
If $p = 1$, the above problem is reduced to the subcritical higher-order fractional Lane-Emden equation with the Navier boundary condition
\[(-\Delta)^s u = u^{\frac{n+2s}{n-2s}-\ep} \text{ in } \Omega \quad \text{and} \quad u = (-\Delta)^{s \over 2} u = 0 \quad \text{for } 1 < s < 2.\]
The main objective of this paper is to deduce the existence of minimal energy solutions,
and to examine their (normalized) pointwise limits provided that $\Omega$ is convex.
As a by-product of our study, a new approach for the existence of an extremal function for the Hardy-Littlewood-Sobolev inequality is provided.
\end{abstract}

\maketitle

\section{Introduction}
Let $\Omega$ be a smooth bounded domain of $\R^n$, $s \in (0,1)$, $n > 2s$ and $(-\Delta)^s$ stand for the fractional Laplacian in $\Omega$,
defined in terms of the spectra of the Laplacian $-\Delta$ in $\Omega$ with zero Dirichlet boundary values on $\pa \Omega$ (often denoted as $(-\Delta|_{\Omega})^s$).
In a series of two papers, we will be interested in the existence and shape of minimal energy solutions to the following nonlinear nonlocal elliptic system
\begin{equation}\label{eq-main}
\begin{cases}
(-\Delta)^s u = v^p &\text{in}~\Omega,
\\
(-\Delta)^s v = u^q &\text{in}~\Omega,
\\
u,\, v > 0 &\text{in } \Omega,
\\
u = v = 0 &\text{in } \R^n \setminus \Omega
\end{cases}
\end{equation}
provided $q \ge p > 2s/(n-2s)$ and
\begin{equation}\label{eq-pq-e}
\ep := \frac{1}{p+1} + \frac{1}{q+1} - {n-2s \over n} > 0
\end{equation}
is small enough.

\medskip
In the last decade, analysis on a variety of nonlocal equations has been to the fore by lots of researchers.
Especially, the local interpretation on the fractional Laplacian in $\R^n$ due to Caffarelli and Silvestre \cite{CS} has brought about a significant development in the study on nonlocal problems,
since it allowed one to utilize well-known techniques for local equations in studying equations of the form $(-\Delta)^s u = f(u)$ where $f: \R \to \R$.
Similar extensions for the fractional Laplacian in bounded domains were soon devised in several works such as Cabr\'e-Tan \cite{CT}, Stinga-Torrea \cite{ST},
Capella-D\'{a}vila-Dupaigne-Sire \cite{CDDS}, Br\"andle-Colorado-de Pablo-S\'anchez \cite{BCDS1} and Tan \cite{T2},
and they have served as fundamental tools in regarding existence, regularity, the Morse index, etc. of solutions to nonlocal equations.
See e.g. \cite{ AC, BCDS2, CaS, CaS2, DDW, DDSV, FFV, FW, JLX,  SV, T} and references therein where such results are obtained with various types of $f$.

A few recent papers have been devoted to study nonlinear systems involving the fractional Laplacians.
For instance, Costa-Miyagaki-Squassina-Yang \cite{CMSY} investigated existence and asymptotic behavior of solutions to the fractional H\'enon systems when the domain is the unit ball.
Leite-Marcos \cite{LM} studied non-variational fractional elliptic system whose form looks like \eqref{eq-main} but the term $(-\Delta)^s v$ is replaced with $(-\Delta)^t v$ for some $t \in (0,1)$.
Some existence, nonexistence and uniqueness of positive viscosity solutions to the fractional Lane-Emden system were obtained in Leite-Montenegro \cite{LM2}.
Moreover, Quaas-Xia \cite{QX, QX2} derived Liouville type results for the fractional Lane-Emden systems in the half and whole spaces, respectively.
See also Quaas-Xia \cite{QX3} which considered existence results for nonlinear cooperative system with gradient terms.

\medskip
The {\it critical Sobolev hyperbola} for \eqref{eq-main} is given by
\begin{equation}\label{eq-sob-h}
\frac{1}{p+1} + \frac{1}{q+1} = \frac{n-2s}{n}.
\end{equation}
We say that a pair $(p,q)$ of positive numbers is {\it subcritical} ({\it critical, supercritical}) if $1/(p+1) + 1/(q+1) > (=, <) (n-2s)/n$, respectively.
In \cite{C}, the first author proved existence of a nontrivial solution to \eqref{eq-main} for the subcritical case and $s \in (0,1)$,
by adapting method of Hulshof-Van der Vorst \cite{HV} and Figueiredo-Felmer \cite{DF} which studied \eqref{eq-main} for $s = 1$.
These results were based on the generalized mountain pass theorem of Benci-Rabinowitz \cite{BR}.

The first objective of this paper is to find a minimal energy solution to \eqref{eq-main}.
For a fixed $s \in (0,1)$ and $(p, q)$ subcritical, the {\it energy functional} of \eqref{eq-main} is defined as
\begin{equation}\label{eq-func}
E_{p,q}(u,v) = \int_{\Omega} (-\Delta)^{\frac{s}{2}} u \cdot (-\Delta)^{\frac{s}{2}} v\, dx - \frac{1}{p+1} \int_{\Omega} v^{p+1} dx - \frac{1}{q+1} \int_{\Omega} u^{q+1} dx
\end{equation}
for $(u,v) \in \mch^s(\Omega) \times \mch^s(\Omega)$ where $\mch^s(\Omega)$ is the fractional Sobolev space whose precise definition is given in \eqref{eq-mvs}.
We say that $(u, v)$ is a {\it minimal energy solution} to \eqref{eq-main}
if it solves \eqref{eq-main} in the sense that it is a critical point of $E_{p,q}$,
and satisfies $E_{p,q}(u,v) \le E_{p,q}(u',v')$ for any other solution $(u',v')$ to \eqref{eq-main}.

On the other hand, it is of definite interest to know the shape of solutions to given nonlinear equations.
In particular, numerous studies have been done when some parameters of elliptic problems tend to a certain critical regime causing a loss of compactness of the given problems.
The second aim of this paper and the contents of the subsequent paper have the same spirit to this direction.
More precisely, we shall concern the asymptotic behavior of minimal energy solutions to the nonlocal problem \eqref{eq-main}
when the domain is convex and the pair of parameters $(p,q)$ gets close to the critical hyperbola.
Our results extend the seminal work of Guerra \cite{G} which conducted asymptotic analysis for \eqref{eq-main} on convex domains in the local case $s = 1$.
It is valuable to mention that the convexity assumption in \cite{G} is  partially removed by the first author in \cite{C3} very recently.
Our work can be also treated as a generalization of the results on the slightly subcritical Lane-Emden equation
\begin{equation}\label{eq-LE}
\begin{cases}
(-\Delta)^s u = u^{\frac{n+2s}{n-2s}-\ep} &\text{in } \Omega,
\\
u > 0 &\text{in } \Omega,
\\
u = 0 &\text{in } \R^n \setminus \Omega.
\end{cases}
\end{equation}
The asymptotic behavior of minimal energy solutions to \eqref{eq-LE} as $\ep \to 0$ was studied
in Han \cite{H} and Rey \cite{R} for $s = 1$, in Chou-Geng \cite{CG} for $s = 2$ (with the Navier boundary condition) and in Choi-Kim-Lee \cite{CKL} for $0 < s < 1$.
As it can be seen in Theorem \ref{thm-bi} below, the case $1 < s < 2$ is covered here as a corollary of our analysis.

\medskip
Throughout the paper, given any $s \in (0,1)$ and $n > 2s$, we fix a value $2s/(n-2s) < p < (n+2s)/(n-2s)$,
and for each $\ep > 0$, denote by $q = q_{\ep} \ge p$ the value satisfying \eqref{eq-pq-e}.
Then it is easy to check that $(p,q_{\ep})$ is subcritical and approaches to the critical hyperbola as $\ep \to 0$.
As it turns out, studying minimal energy solutions to \eqref{eq-main} is closely related to investigating the {\it Hardy-Littlewood-Sobolev inequality}:
There exists a number $C > 0$ depending only on $n, r$ and $\lambda$ such that
\begin{equation}\label{eq-hls}
\left\| |x|^{-\lambda} * f \right\|_{L^{r_1}(\R^n)} \le C \left\| f \right\|_{L^{r_0}(\R^n)} \quad \text{for all } f \in L^{r_0}(\R^n)
\end{equation}
where $0 < \lambda < n$ and $1 < r_0, r_1 < \infty$ satisfies $1/r_0 + \lambda/n = 1 + 1/r_1$.
For the critical pair $(p,q_0)$ satisfying \eqref{eq-sob-h}, we have
$p/(p+1) + (n-2s)/n = 1 + 1/(q_0+1)$,
and so the inequality \eqref{eq-hls} enables us to define the value $S_{p,q_0} \in (0,\infty)$ as
\begin{equation}\label{eq-sob}
S_{p,q_0} = \inf_{f \in C^{\infty}_c(\R^n) \setminus \{0\}} \frac{\|f\|_{L^{\frac{p+1}{p}}(\R^n)}} {g_{n,s} \left\| |x|^{-(n-2s)} * f \right\|_{L^{q_0+1} (\R^n)}}
\end{equation}
where the value $g_{n,s}$ denotes a constant appearing in the Green's function of the Dirichlet fractional Laplacian $(-\Delta)^s$ in \eqref{Green_half}.

\medskip
In our first theorem, the existence of a minimal energy solution for each subcritical pair $(p, q)$ is proved.
Also, we examine the limit of the minimal energy value as $(p, q_{\ep})$ tends to the critical hyperbola (that is, $\ep \to 0$).
\begin{thm}\label{thm-0}
Assume that $\Omega$ is a smooth bounded domain in $\R^n$, $s \in (0,1)$ and $n > 2s$.
Then for any pair $(p,q)$ in the subcritical regime (which satisfies $q \ge p > 2s/(n-2s)$), Eq. \eqref{eq-main} possesses a minimal energy solution $(u,v)$.
Moreover, if we let $(u_{\ep}, v_{\ep})$ be a minimal energy solution to \eqref{eq-main} with $q = q_{\ep}$ for each $\ep > 0$ small, then
\begin{equation}\label{eq-energy}
S_{p,q_\ep}(\Omega) := \frac{ \left\|(-\Delta)^s u_{\ep} \right\|_{L^{\frac{p+1}{p}}(\Omega)}}{\|u_{\ep}\|_{L^{q_{\ep}+1} (\Omega)}} \to S_{p,q_0} \quad \text{as } \ep \to 0
\end{equation}
and
\begin{equation}\label{eq-energy-2}
\lim_{\ep \to 0} E_{p,q_{\ep}}(u_{\ep}, v_{\ep})= {2s \over n} S_{p,q_0}^{n \over 2s}.
\end{equation}
\end{thm}
\noindent To find a solution of \eqref{eq-main}, it may as well write the system as a single equation
\begin{equation}\label{eq-a-1}
\begin{cases}
(-\Delta)^s ((-\Delta)^s u)^{\frac{1}{p}} = u^q &\text{in } \Omega,
\\
u = (-\Delta)^{s} u = 0 &\text{in } \R^n \setminus \Omega
\end{cases}
\end{equation}
by substituting $v = ((-\Delta)^s u)^{1/p}$.
Then, for $s = 1$, the existence of a minimal energy solution to \eqref{eq-main} is easily deduced from
the fact that the embedding $W^{2,(p+1)/p}(\Omega) \hookrightarrow L^{q+1}(\Omega)$ is compact for any subcritical pair $(p,q)$.
Moreover, \eqref{eq-energy} was derived in \cite{W}.

On the other hand, there is a subtle issue in finding a suitable fractional Sobolev space other than the Hilbert spaces $\mathcal{H}^r(\Omega)$ (for some $r > 0$)
in order to derive the existence of minimal energy solutions to the nonlocal problem \eqref{eq-main}.
Besides we should be careful for the zero boundary condition when the function space does not guarantee a sufficient regularity.
For example, the boundary condition might be obscure when $s \in (0, 1/2]$ for the fractional space $\mathcal{H}^s (\Omega)$ since the trace operator is not well-defined in $\mathcal{H}^s(\Omega)$
(refer to \cite[Subsection 2.1]{CDDS} for a further discussion on it).
For these reasons, instead of working on Eq. \eqref{eq-a-1} directly,
we invert the operator $(-\Delta)^s$ to get an integral equation \eqref{eq-w-0} to which we find a solution.
After that, by proving the the regularity of the solution and taking $(-\Delta)^{-s}$ in the both sides of \eqref{eq-w-0},
we will finally obtain a minimal energy solution to \eqref{eq-a-1}.

\medskip
We next prove that minimal energy solutions to \eqref{eq-main} should blow up as $(p, q_{\ep})$ tends to the critical hyperbola.
In addition, we characterize the limit of normalized solutions, which reveals a deep relationship between the system \eqref{eq-main} and the Hardy-Littlewood-Sobolev inequality \eqref{eq-hls}.
\begin{thm}\label{thm-1-2}
Suppose that $\Omega$ is a smooth bounded domain in $\R^n$, $s \in (0,1)$, $n > 2s$, $p > 2s/(n-2s)$
and $(u_{\ep}, v_{\ep})$ is a minimal energy solution to \eqref{eq-main} with $q = q_{\ep}$ for each $\ep > 0$ small.
\begin{enumerate}
\item If we set
\begin{equation}\label{eq-max}
\lambda_{\ep} = \max_{x \in \Omega} u_{\ep}^{1 \over \alpha_{\ep}}(x) = u_{\ep}^{1 \over \alpha_{\ep}}(x_{\ep}) \quad \text {where } \alpha_{\ep} := \frac{2s(p+1)}{pq_{\ep}-1},
\end{equation}
then it holds that
\begin{equation}\label{eq-u-L^inf}
\lambda_{\ep} \textnormal{dist}(x_{\ep}, \pa \Omega) \to \infty.
\end{equation}
\item Define $w_{\ep} = u_{\ep}^{q_{\ep}}$ and
\begin{equation}\label{eq-tw}
\tw_{\ep}(x) = \begin{cases}
\lambda_{\ep}^{-\alpha_{\ep} q_{\ep}} w_{\ep} (\lambda_{\ep}^{-1} x +x_{\ep}) &\text{for } x \in \Omega_{\ep} := \lambda_{\ep}(\Omega- x_{\ep}), \\
0 &\text{for } x \in \R^n \setminus \Omega_{\ep}.
\end{cases}
\end{equation}
Then there exist sequences of positive small numbers $\{\ep_k\}_{k \in \N}$ and functions $\{\tw_{\ep_k}\}_{k \in \N}$,
and a function $W \in L^{(q_0+1)/q_0}(\R^n)$ such that $\ep_k \to 0$ as $k \to \infty$ and
\begin{equation}\label{eq-tw-c}
\lim_{k \to \infty} \int_{\R^n} |\tw_{\ep_k} - W|^{\frac{q_{\ep_k}+1}{q_{\ep_k}}} dx = 0.
\end{equation}
Furthermore, $W$ is a minimizer of the Hardy-Littlewood-Sobolev inequality \eqref{eq-sob}.
\end{enumerate}
\end{thm}
\noindent It is worth to remark that the result above provides a new proof for the existence of minimizer for the Hardy-Littlewood-Sobolev inequality.
Lieb \cite{Li} first proved the existence of the minimizer using the symmetric decreasing rearrangement argument.
Later, Carlen-Lieb \cite{CL} simplified Lieb's proof (refer also to Frank-Lieb \cite{FL}).
Lions \cite{LCC} also proved the existence by applying his concentration-compactness argument.
Our strategy is to construct the minimizer by normalizing a minimal energy solution to \eqref{eq-w-0} and then taking the limit $\ep \to 0$.
Even though our idea is simple and natural, this kind of approach has not appeared in the literature up to the best knowledge of the authors.
It would be worthwhile to extend our approach to cover every possible range of $(p,q,\lambda)$ in the Hardy-Littlewood-Sobolev inequality.

\medskip
Assuming that $\Omega$ is convex and $p \ge 1$ as well as $p > 2s/(n-2s)$, we next investigate
the asymptotic behavior of a solution family $\{(u_{\ep}, v_{\ep})\}_{\ep > 0}$ away from the singularity as $\ep \to 0$.
Let $G$ be Green's function of the fractional Laplacian $(-\Delta)^s$ with zero Dirichlet boundary condition (refer to Subsection \ref{subsec_green}).
For the single equation \eqref{eq-LE}, Choi-Kim-Lee \cite{CKL} proved that if $u_{\ep}$ is a minimal solution to \eqref{eq-LE},
then there exist a point $x_0 \in \Omega$ and a constant $C_0 > 0$ such that
\[u_{\ep} \to 0 \quad \text{and} \quad \|u_{\ep}\|_{L^{\infty}(\Omega)} u_{\ep} \to C_0 G(x,x_0) \quad \text{in } C^{\alpha}(\Omega \setminus \{x_0\})\]
as $\ep \to 0$ for any $\alpha \in (0,2s)$ (in fact, convexity of $\Omega$ is not required here).
For the coupled system \eqref{eq-main}, it will turn out that a similar phenomenon happens if $p \in [n/(n-2s), (n+2s)/(n-2s))$.
However, if $p < n/(n-2s)$, the situation changes drastically and one should introduce a new function $\wtg: \R^n \times \Omega \to \R$ defined by
\begin{equation}\label{eq-tg-1}
\begin{cases}
(-\Delta_x)^s \wtg(x,y) = G^p (x,y) &\text{for } x \in \Omega,
\\
\wtg(x,y) = 0 &\text{for } x \in \R^n \setminus \Omega
\end{cases}
\end{equation}
for each $y \in \Omega$ to handle this case.
In short, the {\it Serrin exponent} $n/(n-2s)$ serves as a threshold for the asymptotic behavior of $\{(u_{\ep}, v_{\ep})\}_{\ep > 0}$ as $\ep \to 0$.
\begin{thm}\label{thm-1}
Suppose that $\Omega$ is a smooth bounded convex domain in $\R^n$, $s \in (0,1)$, $n > 2s$, $p \ge 1$, $p > 2s/(n-2s)$ and $\{ (u_{\ep}, v_{\ep})\}_{\ep > 0}$
is a family of minimal energy solutions to \eqref{eq-main} with $q = q_{\ep}$.
We define the value $\lambda_{\ep} > 0$ and the point $x_{\ep} \in \Omega$ as in \eqref{eq-max}. If
\[\alpha_0 := n - 2s - {n \over p+1} \quad \text{and} \quad \beta_0 := {n \over p+1},\]
then there exists a point $x_0 \in \Omega$ such that $x_{\ep} \to x_0$,
\[\lambda_{\ep}^{\alpha_0} v_{\ep} \to C_1 G(\cdot,x_0) \quad \text{in } C^{0}(\Omega \setminus \{x_0\})\]
and
\[\begin{cases}
\lambda_{\ep}^{\beta_0} u_{\ep} \to C_2 G(\cdot,x_0) &\text{for } \frac{n}{n-2s} < p < \frac{n+2s}{n-2s},\\
\frac{\lambda_{\ep}^{\beta_0}}{\log \lambda_{\ep}} u_{\ep} \to C_3 G(\cdot,x_0) &\text{for } p = \frac{n}{n-2s},\\
\lambda_{\ep}^{p(n-2s-\beta_0)} u_{\ep} \to C_4 \wtg(\cdot,x_0) &\text{for } \frac{2s}{n-2s} < p < \frac{n}{n-2s} \text{ and } p \ge 1
\end{cases}
\text{in } C^{0}(\Omega \setminus \{x_0\})\]
as $\ep \to 0$. Here $C_1, \cdots, C_4 > 0$ are constants depending only on $n,\, s,\, p,\, \Omega$.
\end{thm}
\noindent The quantities $C_1, \cdots, C_4$ are evaluated explicitly in Section \ref{sec-Green}.
Furthermore, the limit of suitably rescaled solutions $(\tu_{\ep}, \tv_{\ep})$ (see \eqref{eq-tuv} for its precise definition) of $(u_{\ep}, v_{\ep})$ is computed in Corollary \ref{cor-tutv-1}.
We note that the convexity of $\Omega$ is used to exclude the possibility that $x_0 \in \pa \Omega$.

One of the key points in the proof of Theorems \ref{thm-1} is to find a global uniform pointwise estimate for $(\tu_{\ep}, \tv_{\ep})$ in $\ep > 0$.
Our main ingredient in the proof will be the Caffarelli-Silvestre extension, the Kelvin transform and the Sobolev inequality.
It is noteworthy that this is the only part where the Caffarelli-Silvestre extension is applied throughout the entire paper.
We will first localize the extended problem \eqref{eq-res} to get Eqs. \eqref{eq-bound-7} or \eqref{eq-bound-8} in the half-ball $B_+^{n+1}(0,r)$ for some small $r > 0$,
and then employ the Brezis-Kato type argument to its integral representation.
This idea works well for the system as well as the scalar equation, and also provides a neater proof compared to \cite{CKL} where the standard Moser iteration argument was used.

\medskip
The next step toward understanding asymptotic behavior of a family $\{ (u_{\ep}, v_{\ep})\}_{\ep > 0}$ of solutions to \eqref{eq-main} with $q = q_{\ep}$
would be to estimate the blow-up rate $\lambda_{\ep}$ in terms of $\ep$ and characterize the blow-up point $x_0 \in \Omega$ as a critical point of a certain function in $\Omega$.
While such a function is expected to the regular part of Green's function $G$ or $\wtg$ defined in \eqref{eq-tg-1}, ascertaining it is quite complicated due to the nonlocal aspect of the problem.
For example, it is not clear how to extract the regular part of $\wtg$ unlike the local case \cite{G}.
It will be fully addressed in the second paper.

\medskip
Once results on the fractional Lane-Emden system \eqref{eq-main} are obtained, we can deduce the same type of conclusions for the higher order fractional Lane-Emden equation as their immediate corollaries.
Indeed, if we set $p = 1$, the system \eqref{eq-main} is reduced to a single problem
\begin{equation}\label{eq-sin}
\begin{cases}
(-\Delta)^s u = u^q &\text{in } \Omega,
\\
u = (-\Delta)^{s \over 2} u = 0 &\text{in } \R^n \setminus \Omega
\end{cases}
\end{equation}
for $s \in (1,2)$, $n > 2s$ and $1 \le q < (n+2s)/(n-2s)$.
Observe that the function $\wtg$ defined in \eqref{eq-tg-1} becomes Green's function of $(-\Delta)^s$ with the Navier boundary condition, that is, a solution of
\[\begin{cases}
(-\Delta_x)^s \wtg(x,y) = \delta_y &\text{in } \Omega,
\\
\wtg(x,y) = (-\Delta)^{s \over 2} \wtg(x,y) = 0 &\text{in } \R^n \setminus \Omega
\end{cases}\]
for each fixed $y \in \Omega$. Here $\delta_y$ is the Dirac delta measure centered at $y$.
Hence the following theorem is a direct consequence of Theorems \ref{thm-0} and \ref{thm-1}.
\begin{thm}\label{thm-bi}
Assume that $\Omega$ is a smooth bounded domain in $\R^n$, $s \in (1,2)$ and $n > 2s$. Then for arbitrary $q \in [1, (n+2s)/(n-2s))$, Eq. \eqref{eq-sin} has a minimal energy solution.
Moreover, if we let $u_{\ep}$ be a minimal energy solution to \eqref{eq-sin} with $q = (n+2s)/(n-2s)-\ep$ for sufficiently small $\ep > 0$, the followings have the validity:
\begin{enumerate}
\item We have
\[\lim_{\ep \to 0} \left[{1 \over 2} \int_{\Omega} \((-\Delta)^{s \over 2} u_{\ep}\)^2 dx
- {1 \over q_{\ep}+1} \int_{\Omega} u_{\ep}^{q_{\ep}+1} dx \right] = S_{1,{n+2s \over n-2s}}^2\]
and
\[S_{1,{n+2s \over n-2s}-\ep}(\Omega) := \frac{ \left\|(-\Delta)^s u_{\ep} \right\|_{L^2(\Omega)}}{\|u_{\ep}\|_{L^{{2n \over n-2s}-\ep}(\Omega)}}
\to S_{1,{n+2s \over n-2s}} \quad \text{as } \ep \to 0.\]

\item If
\[\lambda_{\ep} = \max_{x \in \Omega} u_{\ep}^{2-(2n/s)\ep \over n-2s+2n\ep}(x)
= u_{\ep}^{2-(2n/s)\ep \over n-2s+2n\ep}(x_{\ep}),\]
then it holds that
\[\lambda_{\ep} \textnormal{dist}(x_{\ep}, \pa \Omega) \to \infty.\]
\end{enumerate}
Furthermore, if $\Omega$ is convex, then there exists a point $x_0 \in \Omega$ such that $x_{\ep} \to x_0$ and
\[\lambda_{\ep}^{n-2s \over 2} u_{\ep} \to C_5 \wtg(\cdot,x_0) \quad \text{in } C^0(\Omega \setminus \{x_0\})\]
as $\ep \to 0$. Here $C_5 > 0$ is a constant depending only on $n,\, s,\, p,\, \Omega$.
\end{thm}

The rest of this paper is organized as follows.
In Section \ref{sec_prelim}, we review some preliminaries such as definitions of our fractional Laplacian $(-\Delta)^s$ and Green's function $G$.
In Sections \ref{sec-blow} and \ref{sec-4}, we prove Theorems \ref{thm-0} and \ref{thm-1-2}, respectively.
Section \ref{sec-glo-bdd} is devoted to provide a uniform pointwise bound for a rescaled solution $(\tu_{\ep}, \tv_{\ep})$ in $\ep > 0$.
Based on the result, we show in Section \ref{sec-Green} the convergence of the normalized functions of $\tu_{\ep}$ and $\tv_{\ep}$
to Green's function $G$ or $\wtg$ away from the blow-up point $x_0$, which proves Theorem \ref{thm-1}.
Appendices \ref{app-a} and \ref{app-b} present the proof of two technical lemmas and estimates for $\tv_{\ep}$ under the assumption that $p = n/(n-2s)$, respectively.

\bigskip
\noindent \textbf{Notations.}

\noindent - The letter $z$ represents a variable in the $(n+1)$-dimensional open upper half-space $\R_+^{n+1} := \R^n \times (0, \infty)$.
Also, it is written as $z = (x,t)$ with $x \in \R^n$ and $t > 0$.

\noindent - For any $x \in \R^n$ and $r > 0$, $B^n(x,r)$ and $B^{n+1}_+((x,0),r)$ are the $n$-dimensional ball
and the $(n+1)$-dimensional upper half-ball whose center is $x$ and radius is $r$, respectively.

\noindent - For any domain $B \subset \R^n$, $\chi_B$ denotes the characteristic function of $B$.

\noindent - For any measurable functions $f$ and $g$ in $\R^n$, we define
\[(f*g)(x) = \int_{\R^n} f(x-y) g(y) dy \quad \text{for } x \in \R^n.\]

\noindent - $C > 0$ is a generic constant that may vary from line to line.

\section{Preliminaries}\label{sec_prelim}
\subsection{Spectral Fractional Laplacians}
This subsection is devoted to the precise definition of the operators and spaces which are used throughout the paper.

For a smooth bounded domain $\Omega$ of $\R^n$, we denote by $\{ (\lambda_k, \phi_k) \}_{k=1}^{\infty}$
a sequence of the non-decreasing eigenvalues and corresponding $L^2(\Omega)$-normalized eigenvectors of the Dirichlet Laplacian $-\Delta$ in $\Omega$, solving
\[\begin{cases}
- \Delta \phi_k = \lambda_k \phi_k &\text{in}~ \Omega,\\
\phi_k = 0 &\text{on}~ \pa \Omega.
\end{cases}\]
Then, for $s \in (0,2)$, the {\it fractional Sobolev space} $\mch^s (\Omega)$ is defined as
\begin{equation}\label{eq-mvs}
\mch^s(\Omega) = \left\{ u = \sum_{k=1}^{\infty} a_k \phi_k \in L^2 (\Omega) : \sum_{k=1}^{\infty} a_k^2 \lambda_k^{s} < \infty \right\}.
\end{equation}
Moreover, we let the {\it (spectral) fractional Laplacian} $(-\Delta)^s : \mch^s (\Omega) \to \mch^s(\Omega) \simeq (\mch^s(\Omega))^*$ be
\[(-\Delta)^s \( \sum_{k=1}^{\infty} a_k \phi_k \) = \sum_{k=1}^{\infty} a_k \lambda_k^s \phi_k.\]

In order to utilize the Caffarelli-Silvestre type extension theorems for $s \in (0,1)$, and especially, to consider \eqref{eq-lc-1} instead of \eqref{eq-main},
we introduce a weighted Hilbert space $\mcd^{1,2}(\mcc; t^{1-2s})$ on the half-cylinder $\mcc := \Omega \times (0, \infty)$ which is the completion of
\[ C_{c,L}^{\infty}(\mcc)
:= \left\{ U \in C^{\infty}(\overline{\mcc}) : U = 0 \text{ on } \pa_L\mcc := \pa \Omega \times (0,\infty) \right\}\]
with respect to the norm
\[\|U\|_{\mcd^{1,2}(\mcc; t^{1-2s})} = \( \int_{\mcc} t^{1-2s} |\nabla U|^2 dx dt\)^{1 \over 2}.\]
Recall that it is verified in \cite[Proposition 2.1]{CDDS} that
\[\mch^s(\Omega) = \{u = \text{tr}|_{\Omega \times \{0\}}U: U \in \mcd^{1,2}(\mcc; t^{1-2s})\}.\]

\subsection{Green's functions} \label{subsec_green}
Given any $s \in (0,1)$, we define a function $G_{\R^{n+1}_{+}}$ by
\begin{equation}\label{Green_half}
G_{\R^{n+1}_{+}}((x,t),y) := \frac{ {g_{n,s}}}{|(x-y,t)|^{n-2s}}, \quad g_{n,s} := {\Gamma \({n-2s \over 2}\) \over \pi^{n/2} 2^{2s} \Gamma(s)}
\end{equation}
for every $x, y \in \R^n$ and $t > 0$.
Also, for each fixed $y \in \Omega$, we set $H_{\mcc}(\cdot,y) \in \mcd^{1,2}(\mcc; t^{1-2s})$ as the unique solution to the Dirichlet-Neumann problem
\[\begin{cases}
\text{div}\(t^{1-2s} \nabla H_{\mcc}(\cdot,y) \)= 0 &\text{in } \mcc,
\\
H_{\mcc}(\cdot,y) = \dfrac{g_{n,s}}{|\cdot-(y,0)|^{n-2s}} &\text{on } \pa_L \mcc,
\\
\pa_{\nu}^sH_{\mcc} (\cdot,y) = 0 &\text{on } \Omega \times \{0\}
\end{cases}\]
where
\[\pa_{\nu}^s G_{\mcc}((x,0), y) := - \kappa_s \lim_{t \to 0+} t^{1-2s} {\pa G_{\mcc} \over \pa t}((x,t), y), \quad \kappa_s := {\Gamma(s) \over 2^{1-2s} \Gamma(1-s)},\]
whose existence is guaranteed by a standard minimization argument (see \cite[Lemma 2.2]{CKL}). Then the difference
\begin{equation}\label{eq-green-decom}
G_{\mcc}(z,y) = G_{\R^{n+1}_{+}}(z,y) - H_{\mcc}(z,y) \quad \text{for } z \in \mcc,\ y \in \Omega
\end{equation}
is a solution of
\begin{equation}\label{eq-green}
\begin{cases}
\text{div} (t^{1-2s} \nabla G_{\mcc}(\cdot, y)) = 0 &\text{in } \mcc,
\\
G_{\mcc}(\cdot, y) = 0 &\text{on } \pa_L \mcc,
\\
\pa_{\nu}^s G_{\mcc}(\cdot, y) = \delta_y &\text{on } \Omega \times \{0\},
\end{cases}
\end{equation}
and so it can be called as {\it Green's function $G_{\mcc}$} on the half-cylinder $\mcc$.
By the classical strong maximum principle and the Hopf boundary lemma in \cite[Proposition 4.11]{CaS}, we have
\begin{equation}\label{eq-G}
0 < G_{\mcc}((x,t),y) < {g_{n,s} \over |(x-y,t)|^{n-2s}} \quad \text{for any } x, y \in \Omega \text{ and } t > 0.
\end{equation}

For any $f \in L^{\infty}(\Omega)$, we let
\begin{equation}\label{eq-2-29}
U(z) = \int_{\Omega} G_{\mcc}(z,y) f(y) dy \quad \text{for any } z \in \mcc.
\end{equation}
Employing the Lebesgue dominated convergence theorem, it is not so hard to check that $U \in \mcd^{1,2}(\mcc; t^{1-2s})$ solves
\[\begin{cases}
\text{div} (t^{1-2s} \nabla U) = 0 &\text{in } \mcc,
\\
U = 0 &\text{on } \pa_L \mcc,
\\
\pa_{\nu}^s U = f &\text{on } \Omega \times \{0\}.
\end{cases}\]
Therefore the Caffarelli-Silvestre type extension theorems imply that $u = U(\cdot, 0) \in \mch^{2s}(\Omega)$ satisfies
\[\begin{cases}
(-\Delta)^s u = f &\text{in } \Omega,\\
u = 0 &\text{on } \R^n \setminus \Omega.
\end{cases}\]
In this sense, one can say that $G(x,y) := G_{\mcc}((x,0), y)$ for any $x, y \in \Omega$ is
{\it Green's function} of the fractional Laplacian $(-\Delta)^s$ in $\Omega$ with zero Dirichlet boundary condition so that it holds
\[\begin{cases}
(-\Delta)^s G(\cdot,y) = \delta_y &\text{in } \Omega,
\\
u = 0 &\text{in } \R^n \setminus \Omega
\end{cases}\]
for each $y \in \Omega$, and $H(x,y) := H_{\mcc}((x,0),y)$ for every $x, y \in \Omega$ is the regular part of Green's function $G$ in $\Omega$.
By \cite[Lemma 2.4]{CKL}, $H(x,y) = H(y,x)$ is uniformly bounded for $(x,y) \in \Omega \times K$ where $K$ is an arbitrary compact subset of $\Omega$.

\subsection{Maximum principle}\label{subsec-max}
Here we state the maximum principle which serves as a valuable tool in getting a uniform bound
for dilated solutions $(\tu_{\ep}, \tv_{\ep})$ of \eqref{eq-main} (see Section \ref{sec-glo-bdd}).
Its proof can be found in \cite[Lemma 2.1]{CKL}.
\begin{lem}\label{lem-maximum}
Let $s \in (0,1)$, $\mcd$ be any bounded domain with a piecewise smooth boundary on $\overline{\R^{n+1}_{+}}$ and $U$ a weak solution of
\[\begin{cases}
\textnormal{div}(t^{1-2s} \nabla U) = 0 &\text{in } \mcd,
\\
U(x,t) = F(x,t) &\text{on } \pa_I \mcd := \pa \mcd \cap \R^{n+1}_{+},
\\
\pa_{\nu}^s U(x,0) = 0 &\text{on } \pa_B \mcd := \pa \mcd \cap (\R^n \times \{0\})
\end{cases}\]
for some function $F \in L^{\infty}(\pa_I \mcd)$.
Then we have
\[\sup_{(x,t) \in \mcd} |U(x,t)| \le \sup_{(x,t) \in \pa_I \mcd} |F(x,t)|.\]
\end{lem}
The above lemma allows us show that the equation
\[\begin{cases}
\text{div}(t^{1-2s} \nabla U) = 0 &\text{in } \mcd,
\\
U(x,t) = 0 &\text{on } \pa_I \mcd,
\\
\pa_{\nu}^s U(x,0) = f(x) &\text{on } \pa_B \mcd
\end{cases}\]
admits Green's function $G_\mcd$ such that $G_\mcd(z,y) \le G_{\R^{n+1}_{+}}(z,y)$ for $(z,y) \in \mcd \times \pa_B \mcd$. Thus
\begin{equation}\label{eq-green-rep}
|U(x,0)| \le \int_{\pa_B \mcd} \left|G_\mcd ((x,0),y)\right| |f(y)| dy \le g_{n,s} \int_{\pa_B \mcd} \frac{|f(y)|}{|x-y|^{n-2s}}dy.
\end{equation}
By the Hardy-Littlewood-Sobolev inequality \eqref{eq-hls}, it follows that
\begin{equation}\label{eq-hl}
\|U(\cdot,0)\|_{L^{r_1} (\pa_B \mcd)} \le C \|f\|_{L^{r_0} (\pa_B \mcd)},
\end{equation}
for $1 < r_0 < r_1 < \infty$ such that $1/r_0 - 1/r_1 = 2s/n$.

\section{Proof of Theorem \ref{thm-0}}\label{sec-blow}
First of all, we study the existence of a minimal energy solution to \eqref{eq-main} provided $(p,q)$ is subcritical.
As mentioned before, to find a minimal energy solution to \eqref{eq-main}, it is useful to consider an integral equation
\begin{equation}\label{eq-w-0}
u=(-\Delta)^{-s}((-\Delta)^{-s}(u^q))^{p} \quad \text{in } \Omega,
\end{equation}
which is formally equivalent to \eqref{eq-main}. Here $(-\Delta)^{-s}$ is defined by
\begin{equation}\label{eq-rep}
(-\Delta)^{-s} u(x) = \int_{\Omega} G(x,y) u(y) dy \quad \text{for } x \in \Omega
\end{equation}
(see Subsection \ref{subsec_green} for the definition of $G$).
Furthermore, letting $w=u^q$ in \eqref{eq-w-0} leads to
\begin{equation}\label{eq-w-1}
w^{1 \over q} = (-\Delta)^{-s}((-\Delta)^{-s}w)^p \quad \text{in } \Omega,
\end{equation}
and this is what we actually deal with.

The main virtue of studying in this way is that it allows us to use $L^{(q+1)/q}(\Omega)$ as a proper function space
to find a non-negative solution $w$ of \eqref{eq-w-1} via a variational argument.
By applying an iterative embedding argument, we can prove that $w \in L^{\infty}(\Omega)$, which will enable us to invert the operator $(-\Delta)^{-s}$ in \eqref{eq-w-0}
to certify that $(u, v) = (w^{1/q}, (-\Delta)^{-s} w)$ is a minimal energy solution to \eqref{eq-main}.
\begin{lem}\label{lem-3-1}
Suppose that $(p,q)$ is a subcritical pair (satisfying $q \ge p > 2s/(n-2s)$).
Then \eqref{eq-w-1} possesses a positive solution $w \in L^{(q+1)/q}(\Omega)$ which attains
\begin{equation}\label{eq-Theta}
\Theta_{p,q}(\Omega) := \sup_{f \in L^{\frac{q+1}{q}}(\Omega)\setminus\{0\}} \frac{\|(-\Delta)^{-s} f\|_{L^{p+1}(\Omega)}} {\|f\|_{L^{\frac{q+1}{q}}(\Omega)}}.
\end{equation}
Additionally, it holds that $w \in L^{\infty}(\Omega)$ and
\begin{equation}\label{eq-energy-3}
S_{q_0, p}^{-1} \le \liminf_{\ep \to 0} \Theta_{p,q_{\ep}}(\Omega)
\quad \text{and} \quad
\Theta_{p,q_{\ep}}(\Omega) \le S_{q_0, p}^{-1} |\Omega|^{\frac{1}{q_{\ep} + 1} - \frac{1}{q_0 + 1}}.
\end{equation}
Here $(p, q_0)$ stands for a critical pair which satisfies \eqref{eq-sob-h}.
\end{lem}
\begin{proof}
Since $(p,q)$ is subcritical, we see from Lemma \ref{lem-cpt} that
the embedding $(-\Delta)^{-s}:L^{(q+1)/q}(\Omega) \to L^{p+1}(\Omega)$ is compact,
and so there exists a nontrivial maximizer $w_1$ to \eqref{eq-Theta} such that $\|w_1\|_{L^{(q+1)/q}(\Omega)} = 1$.
Thanks to the fact that $(-\Delta)^{-s}$ has the positive kernel $G$, $|w_1| \in L^{(q+1)/q}(\Omega)$ is also a maximizer to \eqref{eq-Theta} so that $w_1$ can be assumed to be nonnegative.
Now there is a Lagrange multiplier $\mu > 0$ such that
\[\int_{\Omega}((-\Delta)^{-s} w_1)^p (x)(-\Delta)^{-s} \phi (x) dx = \mu \int_{\Omega} w_1^{\frac{1}{q}}(x) \phi (x) dx\]
for any $\phi \in L^{(q+1)/q}(\Omega)$. Because
\[\int_{\Omega}((-\Delta)^{-s} w_1)^p (x) (-\Delta)^{-s} \phi (x) dx = \int_{\Omega}(-\Delta)^{-s}((-\Delta)^{-s} w_1)^p  (x) \phi (x) dx,\]
we find that
\[(-\Delta)^{-s}((-\Delta)^{-s} w_1)^p = \mu w_1^{\frac{1}{q}} \in L^{q+1}(\Omega)
\quad \text{where } \mu = (\Theta_{p,q}(\Omega))^{p+1}.\]
By multiplying $w_1$ by a positive constant, we obtain a nontrivial solution $w$ to \eqref{eq-w-1}, which can be readily checked to be positive and still a maximizer of \eqref{eq-Theta}.

\medskip
We next claim that $w \in L^{\infty}(\Omega)$.
Assume that $w \in L^r(\Omega)$ for some $r \ge (q+1)/q$. Then, owing to estimate \eqref{eq-hl}, it holds that
\begin{itemize}
\item[-] $(-\Delta)^{-s}w \in L^{r_1(\Omega)}$ where $2s/n = 1/r - 1/r_1$,
\item[-] $((-\Delta)^{-s}w)^p \in L^{r_1/p}(\Omega)$,
\item[-] $(-\Delta)^{-s}\(((-\Delta)^{-s}w)^p\) \in L^{r_2}(\Omega)$ where $2s/n = p/r_1 - 1/r_2$,
\item[-] $\((-\Delta)^{-s}\(((-\Delta)^{-s}w)^p\)\)^{q} \in L^{r_3}(\Omega)$ with $r_3 = r_2/q$.
\end{itemize}
By using \eqref{eq-pq-e}, we discover
\[\frac{1}{r}- \frac{1}{r_3} \ge q \left[\frac{1-pq}{q+1} + \frac{2s(p+1)}{n}\right] = (p+1)q \ep.\]
Therefore $1/r - (p+1)q\ep \ge 1/r_3$.
Starting with $r = (q+1)/q$ and applying the above argument repeatedly, we conclude that $w \in L^r(\Omega)$ for any $r \in [1,\infty)$.
Inserting this into \eqref{eq-w-1} gives $w \in L^{\infty}(\Omega)$.

\medskip
Finally, we check \eqref{eq-energy-3}. Without loss of generality, we can assume that $0 \in \Omega$.
Let us fix any ball $B \subset \Omega$ containing the origin and take a nontrivial function $f \in C_c^{\infty}(\R^n)$.
If we write $f_{\delta} = \delta^{-nq_0/(q_0+1)} f(\delta^{-1} \cdot)$, then we have
\begin{align*}
\liminf_{\ep \to 0} \Theta_{p,q_{\ep}}(\Omega)
\ge \liminf_{\ep \to 0} \frac{\|(-\Delta)^{-s} (\chi_B\, f_{\delta})\|_{L^{p+1}(\Omega)}} {\|\chi_B\, f_{\delta}\|_{L^{\frac{q_{\ep}+1}{q_{\ep}}}(\Omega)}}
= \frac{\|(-\Delta)^{-s} (\chi_B\, f_{\delta})\|_{L^{p+1}(\Omega)}} {\|\chi_B\, f_{\delta}\|_{L^{\frac{q_0+1}{q_0}}(\Omega)}}.
\end{align*}
It holds from \eqref{eq-green-decom} that
%\begin{multline*}
%\(\int_{\Omega} \left| \int_B G(x,y) f_{\delta}(y) dy \right|^{p+1} dx \)^{1/(p+1)}
%\\ \le \( \int_{\R^n} \left| \int_{\text{supp}f} {g_{n,s} \over |x-y|^{n-2s}} f(y) dy \right|^{p+1} dx \)^{1/(p+1)}
%- \(C\delta^{(n-2s)p-2s} \|f\|_{L^{\infty}(\Omega)}^{p+1} \int_{\Omega} \left| \int_{\text{supp}f} H(x,\delta y) dy \right|^{p+1} dx \)^{1/(p+1)}
%\end{multline*}
\begin{multline*}
\|(-\Delta)^{-s} (\chi_B\, f_{\delta})\|_{L^{p+1}(\Omega)}
\\ \le g_{n,s} \left\||x|^{-(n-2s)} * f \right\|_{L^{p+1}(\R^n)}
+ C\(\delta^{(n-2s)p-2s} \|f\|_{L^{\infty}(\Omega)}^{p+1} \int_{\Omega} \left| \int_{\text{supp}f} H(x,\delta y) dy \right|^{p+1} dx\)^{1/(p+1)}
\end{multline*}
for small $\delta > 0$, where the second term of the right-hand side tends to 0 as $\delta \to 0$ since $p > 2s/(n-2s)$.
Therefore, by taking $\delta \to 0$, we deduce
\[\liminf_{\ep \to 0} \Theta_{p,q_{\ep}}(\Omega)
\ge \frac{g_{n,s} \left\||x|^{-(n-2s)} * f \right\|_{L^{p+1}(\R^n)}} {\|f\|_{L^{\frac{q_0+1}{q_0}}(\R^n)}}.\]
Since $f$ is arbitrary, we see from \eqref{eq-sob} that the first inequality of \eqref{eq-energy-3} is true.

To derive the other inequality, we first observe with H\"older's inequality that for any nonzero function $f \in L^{(q_{\ep}+1)/q_{\ep}}(\Omega)$,
\[\|f\|_{L^{\frac{q_0+1}{q_0}}(\Omega)} \le \|f\|_{L^{\frac{q_{\ep}+1}{q_{\ep}}}(\Omega)} |\Omega|^{\frac{1}{q_{\ep} + 1} - \frac{1}{q_0 + 1}}.\]
As a result, if we write $\ep = 1/(q_{\ep} + 1) - 1/(q_0 + 1)$
and regard $f$ as a function in $\R^n$ through the null-extension over $\R^n \setminus \overline{\Omega}$, it follows that
\[\frac{\|(-\Delta)^{-s} f\|_{L^{p+1}(\Omega)}}{\|f\|_{L^{\frac{q_{\ep}+1}{q_{\ep}}}(\Omega)}}
\le \frac{\|(-\Delta)^{-s} |f|\|_{L^{p+1}(\Omega)}} {\|f\|_{L^{\frac{q_0+1}{q_0}}(\Omega)}} |\Omega|^{\ep}
\le \frac{g_{n,s} \||x|^{-(n-2s)} * |f|\|_{L^{p+1}(\R^n)}} {\|f\|_{L^{\frac{q_0+1}{q_0}}(\R^n)}} |\Omega|^{\ep} \le S_{q_0,p}^{-1} |\Omega|^{\ep}\]
where we used \eqref{eq-G} in the first and second inequalities. This gives the second inequality of \eqref{eq-energy-3}.
\end{proof}
\begin{lem}\label{lem-min}
Given any subcritical pair $(p,q)$, assume that $w \in L^{\infty}(\Omega)$ is a positive solution of \eqref{eq-w-1} which attains \eqref{eq-Theta}.
Then $(u, v) = (w^{1/q}, (-\Delta)^{-s} w)$ is a solution to \eqref{eq-main}.
\end{lem}
\begin{proof}
Clearly $u$ belongs to $L^{\infty}(\Omega)$, is positive and solves \eqref{eq-w-0}. In particular, $(u,v)$ satisfies
\begin{equation}\label{eq-3-70}
\begin{cases}
u = (-\Delta)^{-s} v^p &\text{in } \Omega,
\\
v = (-\Delta)^{-s} u^q &\text{in } \Omega.
\end{cases}
\end{equation}
Then we observe from \eqref{eq-rep}, \eqref{eq-G} and the discussion in Subsection \ref{subsec_green}
that $v > 0$ in $\Omega$, $v \in L^{\infty}(\Omega) \cap \mch^{2s}(\Omega)$ and $u \in \mch^{2s}(\Omega)$.
Therefore we can apply $(-\Delta)^s$ in each side of \eqref{eq-3-70}, yielding that $(u, v)$  satisfies the first two equalities in \eqref{eq-main}.
Moreover, thanks to \cite[Lemma 2.10]{CDDS} or \cite[Theorem 1.5]{CSt}, we have $u, v \in C^{\alpha}(\overline{\Omega})$ for some $\alpha \in (0,2)$ depending only on $n$ and $s$.
Hence the values of $u$ and $v$ on the boundary $\pa \Omega$ have an appropriate meaning.
Since $G_{\mathcal{C}}(z,y)=0$ for all $z \in \pa_L \mcc$ and $y \in \Omega$ (see \eqref{eq-green}),
we see from Green's representation \eqref{eq-2-29} that $u = v = 0$ in $\R^n \setminus \Omega$.
The proof is completed.
\end{proof}
\noindent Compare our proof with Hang-Yang \cite{HY} where the authors introduced an equivalent maximizing problem involving an integral operator to solve a fourth-order elliptic differential equation.
However, even though the basic idea is similar to ours, the motivation is completely different.
In our situation, the technique was introduced to provide a suitable function space to work with and control the boundary behavior of solutions,
but the authors in \cite{HY} utilized it to guarantee positivity of solutions.

\medskip
The next lemma gives an account that the solution $(u,v)$ to problem \eqref{eq-main} which we obtained in Lemma \ref{lem-min} has indeed the minimal energy among solutions to \eqref{eq-main}.
\begin{lem}\label{lem-a-7}
Assume that $(u,v) \in (\mch^{2s}(\Omega))^2$ is a solution to \eqref{eq-main} and let $w = u^q$. Then we have
\begin{equation}\label{eq-a-7}
E_{p,q}(u,v) = \(1 - \frac{1}{p+1}- \frac{1}{q+1}\) \left[ \frac{\|w\|_{L^{\frac{q+1}{q}}(\Omega)}}{\|(-\Delta)^{-s} w\|_{L^{p+1}(\Omega)}}\right]^{\frac{(p+1)(q+1)}{pq-1}}.
\end{equation}
In particular, if $w$ is a maximizer of \eqref{eq-Theta}, then $(u, v) = (w^{1/q}, (-\Delta)^{-s} w)$ is a minimal solution to \eqref{eq-main}.
\end{lem}
\begin{proof}
Using the self-adjoint property of the operator $(-\Delta)^s$, we get
\begin{equation}\label{eq-a-75}
\int_{\Omega}v^{p+1}(x) dx = \int_{\Omega} (-\Delta)^s u (x) \, v (x) dx
= \int_{\Omega}  u (x) \, (-\Delta)^s v (x) dx = \int_{\Omega} u^{q+1} (x) dx.
\end{equation}
Thus
\begin{equation}\label{eq-a-80}
E_{p,q}(u,v) = \(1- \frac{1}{p+1}-\frac{1}{q+1}\) \int_{\Omega}w^{\frac{q+1}{q}} dx.
\end{equation}
Since $(-\Delta)^{-s} w = v$ in $\Omega$, we also have
\begin{equation}\label{eq-a-74}
\|(-\Delta)^{-s} w\|_{L^{p+1}(\Omega)}^{p+1} = \|w\|_{L^{\frac{q+1}{q}}(\Omega)}^{\frac{q+1}{q}}.
\end{equation}
Combining \eqref{eq-a-80} and \eqref{eq-a-74} gives us the desired result \eqref{eq-a-7}.
\end{proof}

We are ready to complete the proof of Theorem \ref{thm-0}.
\begin{proof}[Proof of Theorem \ref{thm-0}]
We infer from Lemmas \ref{lem-min} and \ref{lem-a-7} that a minimal energy solution to \eqref{eq-main} exists for any subcritical pair $(p,q)$.

Observe that if $(p, q_0)$ is a critical pair, then so is $(q_0, p)$.
Also the Hardy-Littlewood-Sobolev inequality states that a map $\mcl(f) = g_{n,s} |x|^{-(n-2s)} * f$ is the self-adjoint in the sense that
\[\int_{\mathbb{R}^n} \mcl(f)  (x) h(x) dx= \int_{\mathbb{R}^n} f(x)\mcl(h)  (x)  dx\]
for all $f, h \in C_c^{\infty}(\mathbb{R}^n)$. Therefore, using the duality formulation
\begin{align*}
\|\mcl\|_{L^{\frac{p+1}{p}}(\R^n) \to L^{q_0+1}(\R^n)}& := \sup_{ \|f\|_{L^{(p+1)/p}(\R^n)} \le 1} \|\mcl f\|_{L^{q_0 +1}(\R^n)}
\\
&= \sup_{ \|f\|_{L^{(p+1)/p}(\R^n)} \le 1} \( \sup_{ \|h\|_{L^{(q_0+1)/q_0}(\mathbb{R}^n)}\leq 1}\int_{\mathbb{R}^n} \mcl(f) (x) h(x) dx \)
\end{align*}
we find $\|\mcl\|_{L^{\frac{p+1}{p}}(\R^n) \to L^{q_0+1}(\R^n)} = \|\mcl\|_{L^{\frac{q_0+1}{q_0}}(\R^n) \to L^{p+1}(\R^n)}$. Therefore we have $S_{q_0, p} = S_{p,q_0}$.

For each $\ep > 0$ small enough, let $(u_{\ep}, v_{\ep})$ be a minimal energy solution to \eqref{eq-main} with $q = q_{\ep}$, and $w_{\ep} = u_{\ep}^{q_{\ep}}$.
By the definition of $S_{p,q_{\ep}}(\Omega)$, \eqref{eq-a-74} and \eqref{eq-energy-3}, it turns out that
\[\lim_{\ep \to 0} S_{p,q_{\ep}}(\Omega) = \lim_{\ep \to 0} \frac{\|(-\Delta)^{-s} w_{\ep}\|_{L^{p+1}(\Omega)}^p} {\|w_{\ep}\|_{L^{\frac{q_{\ep}+1}{q_{\ep}}}(\Omega)}^{1 \over q_{\ep}}}
= \lim_{\ep \to 0} \frac{\|w_{\ep}\|_{L^{\frac{q_{\ep}+1}{q_{\ep}}}(\Omega)}} {\|(-\Delta)^{-s} w_{\ep}\|_{L^{p+1}(\Omega)}} = S_{q_0, p} = S_{p,q_0},\]
which is \eqref{eq-energy}. Moreover, we infer from \eqref{eq-pq-e}, \eqref{eq-energy}, \eqref{eq-a-80} and \eqref{eq-a-74} that
\begin{equation}\label{eq-a-81}
\lim_{\ep \to 0} E_{p,q_{\ep}}(u_{\ep}, v_{\ep}) = \lim_{\ep \to 0} \(\frac{2s}{n} - \ep\) \int_{\Omega} w_{\ep}^{\frac{q_{\ep}+1}{q_{\ep}}} dx = {2s \over n} S_{p,q_0}^{n \over 2s}.
\end{equation}
Consequently, \eqref{eq-energy-2} has the validity.
\end{proof}

\section{Proof of Theorem \ref{thm-1-2}}\label{sec-4}
In this section, we derive the fact that minimal energy solutions should blow up as $\ep \to 0$ and investigate the limit of their normalizations.
As before, $\{(u_{\ep}, v_{\ep})\}_{\ep > 0}$ denotes a family of solutions to \eqref{eq-main} with $q = q_{\ep}$ satisfying \eqref{eq-energy} and $w_{\ep} = u_{\ep}^{q_{\ep}}$.

\begin{lem}\label{lem-lam}
Let $\lambda_{\ep} > 0$ is the number introduced in \eqref{eq-max}. Then it is true that $\lambda_{\ep} \to \infty$ as $\ep \to 0$.
\end{lem}
\begin{proof}
Assume that $\{u_{\ep}\}_{\ep > 0}$ is uniformly bounded in $\Omega$.
Since $v_{\ep} = (-\Delta)^{-s} u_{\ep}^{q_{\ep}}$, it follows from Lemma \ref{lem-cpt} that there exists a function $v_0 \in L^r(\Omega)$ such that $v_{\ep} \to v_0$ in $L^r(\Omega)$ for any $r > 1$.
In light of \eqref{eq-energy}, it holds that $v_0 \ne 0$ in $\Omega$.
We regard $v_0$ as a function in $\R^n$ by defining $0$ outside $\Omega$. Then it holds that
\[S_{p,q_0}^{-1} \|v_0^p\|_{L^{p+1 \over p}(\R^n)} = S_{p,q_0}^{-1} \|v_0^p\|_{L^{p+1 \over p}(\Omega)}
= \|(-\Delta)^{-s} v_0^p\|_{L^{q_0+1}(\Omega)} < g_{n,s} \left\||x|^{-(n-2s)} * v_0^p \right\|_{L^{q_0+1}(\R^n)}\]
where \eqref{eq-G} was used in the last inequality.
However it violates the definition of $S_{p,q_0}$ in \eqref{eq-sob}.
Hence it should hold that $\lambda_{\ep} \to \infty$ as $\ep \to 0$.
\end{proof}

Define two parameters
\begin{equation}\label{eq-ab}
\alpha_{\ep} = \frac{2s(p+1)}{pq_{\ep}-1} \quad \text{and} \quad \beta_{\ep} = \frac{2s(q_{\ep}+1)}{pq_{\ep}-1},
\end{equation}
and normalize the solutions $(u_{\ep}, v_{\ep})$ to \eqref{eq-main} as
\begin{equation}\label{eq-tuv}
\tu_{\ep} = \lambda_{\ep}^{-\alpha_{\ep}} u_{\ep} (\lambda_{\ep}^{-1} \cdot + x_{\ep}), \quad
\tv_{\ep} = \lambda_{\ep}^{-\beta_{\ep}} v_{\ep} (\lambda_{\ep}^{-1} \cdot + x_{\ep})
\quad \text{in } \Omega_{\ep} = \lambda_{\ep}(\Omega-x_{\ep}).
\end{equation}
They satisfy
\begin{equation}\label{eq-b-20}\begin{cases}
(-\Delta)^s \tu_{\ep} = \tv_{\ep}^{p} &\text{in } \Omega_{\ep},
\\
(-\Delta)^s \tv_{\ep} = \tu_{\ep}^{q_{\ep}} &\text{in } \Omega_{\ep},
\\
\tu_{\ep} = \tv_{\ep} = 0 &\text{in } \R^n \setminus \Omega_{\ep}
\end{cases}
\end{equation}
and
\[\max_{x \in \Omega_{\ep}} \tu_{\ep}(x) = 1 = \tu_{\ep}(0).\]
Also we set $\tw_{\ep} = \tu_{\ep}^{q_{\ep}}$ in $\Omega_{\ep}$. Then it agrees with the definition of $\tw_{\ep}$ given in \eqref{eq-tw}.

If $\lambda_{\ep} \text{dist}(x_{\ep}, \pa \Omega) \to c \in [0, \infty)$ as $\ep \to 0$,
then one may assume that $\Omega_{\ep}$ converges to the upper half-space $\R^n_+$ (rotating the domain $\Omega$ if necessary).
Otherwise, i.e., if $\lim_{\ep \to 0} \lambda_{\ep} \text{dist}(x_{\ep}, \pa \Omega) = \infty$, then $\Omega_{\ep}$ converges to $\R^n$.
\begin{lem}\label{lem-a-4}
Assume that $D$ is the limit set of $\Omega_{\ep}$ as $\ep \to 0$ so that it is either $\R^n_+$ or $\R^n$.
Moreover, let $(-\Delta)^{-s}$ be the operator defined as
\begin{equation}\label{eq-rep-2}
(-\Delta)^{-s} f(x) = g_{n,s} \int_{\R^n} \frac{f(y)}{|x-y|^{n-2s}} dy \quad \text{for } x \in D
\end{equation}
where $f$ is understood to be 0 in $\R^n_- := \R^{n-1} \times (-\infty, 0)$ when $D = \R^n_+$.
Then, passing to a subsequence, the rescaled function $\tw_{\ep}$ converges to a nonzero function $W$
weakly in $L^{\zeta_1}(D)$ and strongly in $C^{\alpha}(D)$ for any $\zeta_1 > (q_0+1)/q_0$ and some $\alpha \in (0,2)$.
Also it holds that
\begin{equation}\label{eq-a-13}
\int_D W^{1 \over q_0}(x) \phi(x) dx \le \int_D (-\Delta)^{-s} \phi(x) ((-\Delta)^{-s} W)^p(x) dx
\end{equation}
for all nonnegative functions $\phi \in C_c^{\infty}(D)$, and
\begin{equation}\label{eq-W}
\|W\|_{L^{\zeta_1}(D)} + \|(-\Delta)^{-s} W\|_{L^{\zeta_2}(D)} < \infty
\end{equation}
for any $\zeta_1 \ge (q_0+1)/q_0$ and $1/\zeta_2 = 1/\zeta_1 + 2s/n$.
\end{lem}
\noindent The proof of this technical lemma is deferred to Appendix \ref{app-a}.

\begin{lem}\label{lem-a-2}
It is valid that
\begin{equation}\label{eq-a-3}
\lim_{\ep \to 0}\lambda_{\ep} \textnormal{dist}(x_{\ep}, \pa \Omega) \to \infty \quad \text{and} \quad \lim_{\ep \to 0} \lambda_{\ep}^{\ep} =1.
\end{equation}
In addition, the pointwise limit $W$ of $\tw_{\ep}$ found in the previous lemma
is a minimizer of the Hardy-Sobolev-Littlewood inequality \eqref{eq-sob} and
\begin{equation}\label{eq-tw-con}
\lim_{\ep \to 0} \int_{\Omega_{\ep}} \tw_{\ep}^{\frac{q_{\ep}+1}{q_{\ep}}}(x) dx = \int_{\R^n} W^{\frac{q_0+1}{q_0}}(x) dx.
\end{equation}
\end{lem}
\begin{proof}
Choose a sequence $\{\phi_k\}_{k \in \N} \subset C_c^{\infty}(D)$ such that $\phi_k \to W$ in $L^{(q_0+1)/q_0}(D)$ as $k \to \infty$.
By the Hardy-Littlewood-Sobolev inequality \eqref{eq-hls}, it follows that $(-\Delta)^{-s}\phi_k \to (-\Delta)^{-s}W$ in $L^{p+1}(D)$.
Hence we are able to take $\phi = \phi_k$ in \eqref{eq-a-13} and send $k \to \infty$ to get
\[\int_D W^{\frac{q_0+1}{q_0}}(x) dx \le \int_D ((-\Delta)^{-s} W)^{p+1}(x) dx,\]
or equivalently,
\begin{equation}\label{eq-a-17}
\frac{1}{\|W\|_{L^{\frac{q_0+1}{q_0}}(D)}^{\gamma_0}} \le \frac{g_{n,s} \||x|^{-(n-2s)} * W\|_{L^{p+1}(D)}}{\|W\|_{L^{\frac{q_0+1}{q_0}}(D)}}
\quad \text{where } \gamma_0 := \frac{pq_0-1}{q_0(p+1)} > 0.
\end{equation}

If $D = \R_+^n$, we let $W = 0$ on $\R_-^n$. Then we see from \eqref{eq-sob} that
\begin{equation}\label{eq-a-77}
\frac{g_{n,s} \||x|^{-(n-2s)} * W\|_{L^{p+1}(D)}} {\|W\|_{L^{\frac{q_0+1}{q_0}}(D)}}
\le \frac{g_{n,s} \||x|^{-(n-2s)} * W\|_{L^{p+1}(\R^n)}}{\|W\|_{L^{\frac{q_0+1}{q_0}}(\R^n)}} \le S_{q_0,p}^{-1}
\end{equation}
where the first inequality is indeed an equality if and only if $D = \R^n$.

Thanks to \eqref{eq-tw}, we have
\begin{equation}\label{eq-a-34}
\int_{\Omega} w_{\ep}^{\frac{q_{\ep}+1}{q_{\ep}}} (x) dx = \lambda_{\ep}^{(q_{\ep}+1)\alpha_{\ep}-n} \int_{\Omega_{\ep}} \tw_{\ep}^{\frac{q_{\ep}+1}{q_{\ep}}}(x) dx.
\end{equation}
Also, \eqref{eq-pq-e} and interior regularity property of \eqref{eq-b-20} (see e.g. \cite{CDDS, CaS, CSt}) tell us that
\[(q_{\ep}+1) \alpha_{\ep} - n = {n^2 \over 2s} \ep + O(\ep^2)\]
and $\tw_{\ep} \to W$ a.e. in $D$ along a subsequence, respectively.
Considering \eqref{eq-a-81}, we deduce
\begin{equation}\label{eq-a-35}
S_{q_0,p}^{n \over 2s} \ge \liminf_{\ep \to 0}  \int_{\Omega_{\ep}} \tw_{\ep}^{\frac{q_{\ep}+1}{q_{\ep}}} (x) dx
\ge \(\liminf_{\ep \to 0} \lambda_{\ep}^{{n^2\ep \over 2s}}\) \int_{D} W^{\frac{q_0+1}{q_0}}(x) dx = \mca \, \|W\|_{L^{\frac{q_0+1}{q_0}}(D)}^{\frac{q_0+1}{q_0}}
\end{equation}
where we set $\mca =\liminf_{\ep \to 0} \lambda_{\ep}^{n^2\ep/(2s)} \ge 1$. This allows us to find that
\[ \mca^{q_0 \gamma_0 \over q_0+1} S_{q_0,p}^{-1}= \mca^{q_0 \gamma_0 \over q_0+1} S_{q_0,p}^{-{n q_0 \gamma_0 \over 2s (q_0+1)}}
\le \frac{1}{\|W\|_{L^{\frac{q_0+1}{q_0}}(D)}^{\gamma_0}} \le S_{q_0,p}^{-1},\]
where \eqref{eq-a-17} and \eqref{eq-a-77} are used for the last inequality.
Since $\mca \ge 1$, the above computation implies that $\mca = 1$ and all the inequalities in \eqref{eq-a-17}, \eqref{eq-a-77} and \eqref{eq-a-35} should be equal.
Therefore $D = \R^n$ (from which we infer that $\lambda_{\ep} \text{dist}(x_{\ep}, \pa \Omega) \to \infty$ as $\ep \to 0$),
$W$ attains \eqref{eq-sob}, and \eqref{eq-tw-con} is true. The proof of the lemma is completed.
\end{proof}

\begin{lem}\label{lem-tutv-2}
Let $\mcu = W^{1/q_0}$. It holds that
\[\lim_{\ep \to 0} \int_{\R^n} |\tw_{\ep}(x) - W(x)|^{\frac{q_{\ep} + 1}{q_{\ep}}} dx = 0\quad \textrm{and}\quad \lim_{\ep \to 0} \int_{\R^n} |\tu_{\ep}(x) - \mcu(x)|^{q_{0}+1} dx =0.\]
\end{lem}
\begin{proof}
As in the proof of the Brezis-Lieb lemma, it can be checked that
\[\lim_{\ep \to 0} \int_{\R^n} \( |\tw_{\ep}(x)|^{\frac{q_{\ep} + 1}{q_{\ep}}} dx - |\tw_{\ep}(x) - W(x)|^{\frac{q_{\ep} + 1}{q_{\ep}}}  \) dx = \int_{\R^n} |W(x)|^{\frac{q_0+1}{q_0}} dx.\]
Combining this with \eqref{eq-tw-con}, we obtain the first estimate.
For the second estimate, we note that since $\tilde{w}_{\ep}$ and $W$ are uniformly $L^{\infty}$ bounded,
\begin{align*}
|\tilde{u}_{\ep} (x) - \tu_{\ep} (x)|^{q_0 +1} &\leq |\tilde{w}_{\ep}^{\frac{1}{q_{\ep}}} - W^{\frac{1}{q_0}} |^{q_0 +1}
\\
&\leq  C|\tilde{w}_{\ep}^{\frac{1}{q_{\ep}}} - W^{\frac{1}{q_{\ep}}} |^{q_0 +1}+C|W^{\frac{1}{q_{\ep}}} - W^{\frac{1}{q_0}} |^{q_0 +1}
\\
&\leq C |\tilde{w}_{\ep}^{\frac{1}{q_{\ep}}} - W^{\frac{1}{q_{\ep}}} |^{q_{\ep} +1}+C|W^{\frac{1}{q_{\ep}}} - W^{\frac{1}{q_0}} |^{q_0 +1}
\\
&\leq C |\tilde{w}_{\ep} - W |^{\frac{q_{\ep} +1}{q_{\ep}}}+C|W^{\frac{1}{q_{\ep}}} - W^{\frac{1}{q_0}} |^{q_0 +1}.
\end{align*}
Therefore, using the first equality of the lemma and the dominated convergence theorem, we deduce that
\begin{align*}
&\lim_{\ep \rightarrow 0}\int_{\mathbb{R}^n} |\tilde{u}_{\ep} (x) - \tu_{\ep} (x)|^{q_0 +1} dx
\\
&\leq C \int_{\mathbb{R}^n}|\tilde{w}_{\ep} - W |^{\frac{q_{\ep} +1}{q_{\ep}}}dx +C\int_{\mathbb{R}^n}|W^{\frac{1}{q_{\ep}}} - W^{\frac{1}{q_0}} |^{q_0 +1}dx =0,
\end{align*}
which proves the second estimate.
\end{proof}

\begin{proof}[Proof of Theorem \ref{thm-1-2}]
Estimates \eqref{eq-u-L^inf} and \eqref{eq-tw-c} follow from Lemmas \ref{lem-a-2} and \ref{lem-tutv-2}, respectively.
Lemma \ref{lem-a-2} also implies that $W$ is a minimizer of the Hardy-Littlewood-Sobolev inequality \eqref{eq-sob}.
\end{proof}

\section{Asymptotic Behavior of Rescaled Solutions}\label{sec-glo-bdd}
In this section and Section \ref{sec-Green}, we always assume that $\Omega$ is convex.
Also, we write $(u_{\ep}, v_{\ep})$ and $(\tu_{\ep}, \tv_{\ep})$
to denote a minimal energy solution to \eqref{eq-main} with $q = q_{\ep}$ and its rescaling given by \eqref{eq-tuv}.

\subsection{The location of the concentration point $x_0$}
Here we are concerned with uniform boundedness of solutions $(u,v)$ to \eqref{eq-main} near the boundary $\pa \Omega$, provided that $(p,q)$ is subcritical.
For any sufficiently small $\delta > 0$, let
\[\mco(\Omega, \delta) = \{x \in \Omega: \text{dist}(x,\pa \Omega) < \delta\}\]
and
\[\mci(\Omega, \delta) = \{x \in \Omega: \text{dist}(x,\pa \Omega) > \delta\}.\]
\begin{lem}\label{lem-bb}
Suppose that $\Omega$ is convex and $(p,q)$ is a subcritical pair such that $p, q > 1 + \eta$ for a fixed value $\eta > 0$.
Then there are positive constants $C$ large and $\delta$ small depending only on $n,\, s,\, \Omega,\, \eta$ such that
\[\sup_{x \in \mco(\Omega, \delta)} \left[u (x) + v(x)\right] \le C\]
for any solution $(u,v)$ to \eqref{eq-main}.
\end{lem}
\begin{proof}
Assume that $(\lambda_1, \phi_1)$ is the first eigenpair of the Dirichlet Laplacian $-\Delta$ in $\Omega$. Then we have
\[\lambda_1^s \int_{\Omega} u \phi_1 dx = \int_{\Omega} v^p \phi_1 dx
\quad \text{and} \quad
\lambda_1^s \int_{\Omega} v \phi_1 dx = \int_{\Omega} u^q \phi_1 dx.\]
Applying Jensen's inequality in the right-hand sides (which is possible due to the condition $p, q > 1$), we get
\[{\lambda_1}^s \int_{\Omega} u \phi_1 dx \ge C \(\int_{\Omega} v \phi_1 dx\)^p
\quad \text{and} \quad
{\lambda_1}^s \int_{\Omega} v \phi_1 dx \ge C \(\int_{\Omega} u \phi_1 dx\)^q,\]
which easily yields
\begin{equation}\label{eq-bb}
\int_{\Omega} v \phi_1 dx \le C \quad \text{and} \quad \int_{\Omega} u \phi_1 dx \le C.
\end{equation}
Since $\Omega$ is convex, one can find small constants $\delta_0 > 0$ and $m_0 > 0$ such that
for each point $x \in \pa \Omega$ there exists an open connected subset $Q_x$ of the unit sphere $\ms^{n-1}$ satisfying
\begin{itemize}
\item[-] $|Q_x| > m_0$ and $A_x := \{ x + tw : 0 \le t \le \delta_0, w \in Q_x \} \subset \Omega$;
\item[-] For each 3-tuple $(x, t, w)$ such that $x + tw \in A_x$,
let $P_{x,t,w}$ be a plane which passes through $x + tw$ and has $w$ as its normal vector.
Suppose that it divides $\Omega$ into two parts $\Omega_1$ and $\Omega_2$ where $x \in \pa \Omega_1$.
Then the reflection of $\Omega_1$ with respect to the plane $P_{x,t,w}$ is contained in $\Omega_2$.
\end{itemize}
Then the moving plane argument guarantees
\[u(x+t_1 w) \le u(x+ t_2 w) \quad \text{and} \quad v(x+t_1 w) \le v(x+ t_2 w) \quad \text{for all } 0 \le t_1 \le t_2 \le \delta_0 \text{ and } w \in Q_x.\]
Consequently, there exist small numbers $\delta_1 > 0$ and $m_1 > 0$ such that for each $x \in \mco(\Omega, \delta_1)$,
one can select $B_x \subset \Omega$ satisfying
\begin{itemize}
\item[-] $|B_x| > m_1$ and $B_x \subset \mci(\Omega, \delta_1/2)$;
\item[-] $u(y) > u(x)$ and $v(y) > v(x)$ for every $y \in B_x$.
\end{itemize}
As a result, we have
\[u(x) \le \frac{1}{|B_x|} \int_{B_x} u(y) dy \le \frac{1}{m_1} \int_{\mathcal{I}(\Omega, \delta_1/2)} u(y) dy \quad \text{for all } x \in \mco(\Omega,\delta_1)\]
and the same inequalities for $v$.
The $L^1$-bounds \eqref{eq-bb} of $u$ and $v$ give the desired uniform bounds of $u$ and $v$ in $\mco(\Omega,\delta_1)$.
\end{proof}

\subsection{Global boundedness of rescaled solutions}
The main goal of this subsection is to show uniform pointwise boundedness of the rescaled solutions $(\tu_{\ep}, \tv_{\ep})$ with respect to $\ep > 0$.
For this aim, we first apply the Caffarelli-Silvestre type extension, the Kelvin transform and a certain localization technique.
Then we utilize a variant of the Brezis-Kato type argument.
%Our proof differs from that of the analogous lemma (Lemma 2.3) for the local problem $s = 1$ in \cite{G}.
%It seems valuable to compare them.

\medskip
For a given minimal energy solution $(u_{\ep}, v_{\ep})$ to \eqref{eq-main} with $q = q_{\ep}$,
define its \textit{$s$-harmonic extension} $(U_{\ep}, V_{\ep}) \in (\mcd^{1,2}(\mcc; t^{1-2s}))^2$ as the solution of the degenerate local elliptic system
\begin{equation}\label{eq-lc-1}
\begin{cases}
\text{div}(t^{1-2s} \nabla U_{\ep}) = \text{div}(t^{1-2s} \nabla V_{\ep}) = 0 &\text{in } \mcc,\\
U_{\ep}, V_{\ep} > 0 &\text{in } \mcc,\\
U_{\ep} = V_{\ep} = 0 &\text{on } \pa_L \mcc,\\
U_{\ep} = u_{\ep}, V_{\ep} = v_{\ep} &\text{on } \Omega \times \{0\}
\end{cases}
\end{equation}
where $\mcc = \Omega \times (0,\infty)$ and $\pa_L \mcc = \pa \Omega \times (0,\infty)$.
By \cite{CT, ST, CDDS, BCDS1, T2}, it holds that
\[\pa_{\nu}^s U_{\ep} = (-\Delta)^s u_{\ep} = v_{\ep}^p
\quad \text{and} \quad
\pa_{\nu}^s V_{\ep} = (-\Delta)^s v_{\ep} = u_{\ep}^{q_{\ep}} \quad \text{on } \Omega \times \{0\}.\]
We also rescale them as
\[\wtu_{\ep}(z) = \lambda_{\ep}^{-\alpha_{\ep}} U_{\ep} (\lambda_{\ep}^{-1} z + (x_{\ep},0))
\quad \text{and} \quad
\wtv_{\ep}(z) = \lambda_{\ep}^{-\beta_{\ep}} V_{\ep} (\lambda_{\ep}^{-1} z + (x_{\ep},0))\]
for $z \in \mcc_{\ep} := \Omega_{\ep} \times (0,\infty) = \lambda_{\ep}(\Omega-x_{\ep}) \times (0,\infty)$
where the definition of $\alpha_{\ep}$, $\beta_{\ep}$ and $\lambda_{\ep}$ can be found in \eqref{eq-ab} and \eqref{eq-max}.
Then $(\wtu_{\ep}, \wtv_{\ep})$ satisfies
\begin{equation}\label{eq-tUV}
\begin{cases}
\text{div}(t^{1-2s} \nabla \wtu_{\ep}) = \text{div}(t^{1-2s} \nabla \wtv_{\ep}) = 0 &\text{in } \mcc_{\ep} := \Omega_{\ep} \times (0,\infty),
\\
\wtu_{\ep}, \wtv_{\ep} > 0 &\text{in } \mcc_{\ep},
\\
\wtu_{\ep} = \wtv_{\ep} = 0 &\text{on } \pa_L \mcc_{\ep} := \pa \Omega_{\ep} \times (0, \infty),
\\
\pa_{\nu}^s \wtu_{\ep} = \tv_{\ep}^p, \pa_{\nu}^s \wtv_{\ep} = \tu_{\ep}^{q_{\ep}} &\text{on } \Omega_{\ep} \times \{0\}.
\end{cases}
\end{equation}
Here $(\tu_{\ep}, \tv_{\ep})$ is the rescaled solution given by \eqref{eq-tuv}.

\begin{prop}\label{lem-gu}
Suppose that $(\wtmcu, \wtmcv)$ is a pair of smooth functions in $\R^n$ which satisfies
\begin{equation}\label{eq-decay}
\wtmcv(r) \simeq r^{-(n-2s)}
\quad \text{and} \quad
\begin{cases}
\wtmcu(r) \simeq r^{-(n-2s)} &\text{for } \frac{n}{n-2s} < p < \frac{n+2s}{n-2s},
\\
\wtmcu(r) \simeq r^{-(n-2s)} \log r &\text{for } p = \frac{n}{n-2s},
\\
\wtmcu(r) \simeq r^{-(p(n-2s)-2s)} &\text{for } \frac{2s}{n-2s} < p < \frac{n}{n-2s} \text{ and } p \ge 1.
\end{cases}
\end{equation}
Here $r := |x|$ and $f(r) \simeq g(r)$ means that $c^{-1} f(r) \le g(r) \le c f(r)$ as $r \to \infty$ for some $c > 1$.
Then there exists a constant $C > 0$ depending only on $n$, $s$, $p$ and $\Omega$ such that
\begin{equation}\label{eq-bound}
\tu_{\ep} \le C \wtmcu \quad \text{and} \quad \tv_{\ep} \le C \wtmcv \quad \text{in } \Omega_{\ep}
\end{equation}
for all $\ep > 0$ small.
\end{prop}
\noindent We first prepare several preliminary lemmas which are needed to solve Proposition \ref{lem-gu}.

\medskip
Let $\kappa$ be the inversion in $\R_+^{n+1} \setminus \{0\}$, that is, $\kappa(z)= z/|z|^2$ for all $z \in \R_+^{n+1} \setminus \{0\}$.
Moreover we define the \textit{Kelvin transform} $(F_{\ep}, G_{\ep})$  of $(\wtu_{\ep}, \wtv_{\ep})$ by
\begin{equation}\label{eq-bound-6}
F_{\ep}(z) = |z|^{-(n-2s)} \wtu_{\ep}(\kappa(z)) \quad \text{and} \quad G_{\ep}(z) = |z|^{-(n-2s)} \wtv_{\ep}(\kappa(z)) \quad \text{for } z \in \kappa(\mcc_{\ep})
\end{equation}
(do not confuse the above $G_{\ep}$ with a rescaled Green's function defined in \eqref{eq-G_ep}).
Then a direct computation using \eqref{eq-tUV} shows that
\begin{equation}\label{eq-res}
\begin{cases}
\text{div}(t^{1-2s} \nabla F_{\ep}) = \text{div}(t^{1-2s} \nabla G_{\ep}) = 0 &\text{in } \kappa(\mcc_{\ep}),
\\
F_{\ep} = G_{\ep} = 0 &\text{on } \kappa(\pa_L \mcc_{\ep}),
\\
\pa_{\nu}^s F_{\ep}(x,0) = |x|^{-(n+2s) + (n-2s)p} G_{\ep}^p(x,0) &\text{on } \kappa (\Omega_{\ep}) \times \{0\},\\
\pa_{\nu}^s G_{\ep}(x,0) = |x|^{-(n+2s) + (n-2s)q_{\ep}} F_{\ep}^{q_\ep}(x,0) &\text{on } \kappa (\Omega_{\ep}) \times \{0\}.
\end{cases}
\end{equation}
To prove the proposition, it is sufficient to compute the blow-up rate of $\{(F_{\ep}, G_{\ep})\}_{\ep > 0}$ near the origin.
For this aim, it is convenient to localize $(F_{\ep}, G_{\ep})$ in the following way:
Take a small $r > 0$ and define $\whf_{\ep}$ and $\whg_{\ep}: B^{n+1}_+(0,r) \to \R$ to solve
\begin{equation}\label{eq-bound-7}
\begin{cases}
\text{div}(t^{1-2s} \nabla \whf_{\ep}) = 0 &\text{in } B^{n+1}_+(0,r),
\\
\whf_{\ep} = 0 &\text{on } B^{n+1}_+(0,r) \cap \pa \kappa(\mcc_{\ep}),
\\
\whf_{\ep} = 0 &\text{on } \pa B^{n+1}_+(0,r) \cap \R^{n+1}_+,
\\
\pa_{\nu}^s \whf_{\ep}(x,0) = |x|^{-(n+2s) + (n-2s)p} G_{\ep}^p(x,0) &\text{on } (B^n(0,r) \cap \kappa (\Omega_{\ep})) \times \{0\}
\end{cases}
\end{equation}
and
\begin{equation}\label{eq-bound-8}
\begin{cases}
\text{div}(t^{1-2s} \nabla \whg_{\ep}) = 0 &\text{in } B^{n+1}_+(0,r),
\\
\whg_{\ep} = 0 &\text{on } B^{n+1}_+(0,r) \cap \pa \kappa(\mcc_{\ep}),
\\
\whg_{\ep} = 0 &\text{on } \pa B^{n+1}_+(0,r) \cap \R^{n+1}_+,
\\
\pa_{\nu}^s \whg_{\ep}(x,0) = |x|^{-(n+2s) + (n-2s)q_{\ep}} F_{\ep}^{q_{\ep}}(x,0) &\text{on } (B^n(0,r) \cap \kappa (\Omega_{\ep})) \times \{0\}.
\end{cases}
\end{equation}
Then the differences $R_1 := F_{\ep} - \whf_{\ep}$ and $R_2 := G_{\ep} - \whg_{\ep}$ satisfy
\[\begin{cases}
\text{div}(t^{1-2s} \nabla R_1) = \text{div}(t^{1-2s} \nabla R_2) = 0 &\text{in } B^{n+1}_+(0,r),
\\
R_1 = R_2 = 0 &\text{on } B^{n+1}_+(0,r) \cap \pa \kappa(\mcc_{\ep}),
\\
R_1 = F_{\ep},\, R_2 = G_{\ep} &\text{on } \pa B^{n+1}_+(0,r) \cap \R^{n+1}_+,
\\
\pa_{\nu}^s R_1 = \pa_{\nu}^s R_2 = 0 &\text{on } (B^n(0,r) \cap \kappa (\Omega_{\ep})) \times \{0\}.
\end{cases}\]
Since $\{(\tu_{\ep},\, \tv_{\ep})\}_{\ep > 0}$ are uniformly bounded in $\Omega_{\ep}$,
so is $\{(\wtu_{\ep}, \wtv_{\ep})\}$ in $\mathcal{C}_{\ep}$ by \eqref{eq-2-29} and \eqref{eq-G}.
Therefore, in view of the definition \eqref{eq-bound-6} of $(F_{\ep}, G_{\ep})$, we can find a constant $C > 0$ independent of $\ep > 0$ such that
\[|F_{\ep}| + |G_{\ep}| \le C \quad \text{on } \pa B^{n+1}_+(0,r) \cap \R^{n+1}_+.\]
Combining this with Lemma \ref{lem-maximum}, we observe that $R_1$ and $R_2$ are uniformly bounded in $B^{n+1}_+(0,r) \cap \kappa(\mcc_{\ep})$.
Hence there exists a constant $C > 0$ independent of $\ep > 0$ such that
\begin{equation}\label{eq-bound-1}
\whf_{\ep} - C \le F_{\ep} \le \whf_{\ep} + C \quad \text{and} \quad \whg_{\ep} - C \le G_{\ep} \le \whg_{\ep} + C \quad \text{in } B^{n+1}_+(0,r) \cap \kappa(\mcc_{\ep})
\end{equation}
for any $\ep > 0$ small enough.
In this reason, it suffices to estimate the functions $\whf_{\ep}$ and $\whg_{\ep}$ instead.

In what follows, we denote by $f(x)*g(x)$ the convolution $(f*g)(x)$.
\begin{lem}
Let us regard $\whf_{\ep}$ and $\whg_{\ep}$ as functions in $\R^{n+1}_+$ by taking 0 outside of their actual domains.
Also we set the operators
\[K_{\ep} f = |x|^{-(n-2s)} * \( |x|^{-\theta(q_0)(q_0+1)} \whf_{\ep}^{q_0-1}(x,0) f(x) \)\]
and
\[L_{\ep} f = |x|^{-(n-2s)} * \( |x|^{-\theta(p)(p+1)} \whg_{\ep}^{p-1}(x,0) f(x)\)\]
for a function $f$ in $\R^n$ where
\begin{equation}\label{eq-theta}
\theta(p) := {(n+2s) - (n-2s)p \over p+1} > 0 \quad \text{and} \quad \theta(q_0) := {(n+2s) - (n-2s) q_0 \over q_0+1} \underset{\eqref{eq-pq-e}}{=} - \theta(p),
\end{equation}
Then we have
\begin{equation}\label{eq-bound-3}
\whg_{\ep}(x,0) \le C \(K_{\ep} L_{\ep} \whg_{\ep}(x,0) + 1\) \quad \text{for every } x \in B^n(0,r) \cap \kappa (\Omega_{\ep}).
\end{equation}
\end{lem}
\begin{proof}
If we set $\chi_{B}(x) = \chi_{B^n(0,r)}(x)$, we see from \eqref{eq-green-rep}, \eqref{eq-a-3} and \eqref{eq-bound-7}-\eqref{eq-bound-1} that
\begin{equation}\label{eq-bound-11}
\whf_{\ep}(x,0) \le \chi_B(x) \cdot \left[ |x|^{-(n-2s)} * \left\{ |x|^{-(n+2s)+(n-2s)p} \(\whg_{\ep}^p(x,0) + C\chi_B(x)\) \right\} \right]
\end{equation}
and
\begin{equation}\label{eq-bound-12}
\whg_{\ep}(x,0) \le \chi_B(x) \cdot \left[ |x|^{-(n-2s)} * \left\{ |x|^{-(n+2s)+(n-2s)q_0} \(\whf_{\ep}^{q_0}(x,0) + C\chi_B(x)\) \right\} \right]
\end{equation}
for every $x \in B^n(0,r) \cap \kappa(\Omega_{\ep})$. In addition, after some elementary computations using the fact that $q_0 > n/(n-2s)$ for small $\ep > 0$, we obtain
\begin{equation}\label{eq-bound-2}
|x|^{-(n-2s)} * \(|x|^{-(n+2s)+(n-2s)p} \chi_B(x)\) \le C \eta_p(x)
\end{equation}
and
\begin{equation}\label{eq-bound-9}
|x|^{-(n-2s)} * \(|x|^{-(n+2s)+(n-2s)q_0} \chi_B(x)\) \le C
\end{equation}
where
\[\eta_p(x) := \begin{cases}
1 &\text{for } \frac{n}{n-2s} < p < \frac{n+2s}{n-2s},
\\
\log|x| &\text{for } p = \frac{n}{n-2s},
\\
|x|^{-n+(n-2s)p} &\text{for } \frac{2s}{n-2s} < p < \frac{n}{n-2s}
\end{cases}\]
for $x \in \R^n \setminus \{0\}$. Using these computations, we infer from \eqref{eq-bound-11} and \eqref{eq-bound-12} that
\[\whf_{\ep}(x,0) \le \chi_B(x) \cdot \(L_{\ep} \whg_{\ep}(x,0) + C \eta_p (x)\) \quad \text{and} \quad \whg_{\ep}(x,0) \le \chi_B(x) \cdot \(K_{\ep} \whf_{\ep}(x,0) + C\).\]
Consequently,
\begin{align*}
\whg_{\ep}(x,0)
&\le C |x|^{-(n-2s)} * \left[ |x|^{-\theta(q_0)(q_0+1)} \whf_{\ep}^{q_0}(x,0) \right] + C
\\
&\le C |x|^{-(n-2s)} * \left[ |x|^{-\theta(q_0)(q_0+1)} \( \whf_{\ep}^{q_0-1}(x,0) L_{\ep}\whg_{\ep}(x) + \eta_p^{q_0}(x) \chi_B(x) \) \right] + C
\\
&\le C \(K_{\ep} L_{\ep} \whg_{\ep}(x,0) + 1\)
\end{align*}
for any $x \in B^n(0,r) \cap \kappa (\Omega_{\ep})$. The lemma is proved.
\end{proof}

For the sake of simplicity, we denote $L^d = L^d(B^n(0,r) \cap \kappa (\Omega_{\ep}))$ for each $d \in [1, \infty)$ and $r \in (0,\infty)$ in the below.
\begin{lem}\label{eq-AB}
We consider the functions
\[A_{\ep}(x) := |x|^{-\theta(q_0)(q_0-1)} \whf_{\ep}^{q_0-1}(x,0)
\quad \text{and} \quad
B_{\ep}(x) := |x|^{-\theta(p)(p-1)} \whg_{\ep}^{p-1}(x,0)\]
where $x \in B^n(0,r) \cap \kappa (\Omega_{\ep})$. These functions satisfy
\begin{equation}\label{eq-A-ep}
\sup_{\ep > 0} \( \|A_{\ep}\|_{L^{\frac{q_0 +1}{q_0 -1}}} + \|B_{\ep}\|_{L^{\frac{p+1}{p-1}}} \) < \infty.
\end{equation}
In addition,  for any $\delta >0$, we can choose $r > 0$ so small that
\begin{equation}\label{eq-A-ep-d}
\sup_{\ep > 0} \|A_{\ep}\|_{L^{\frac{q_0 +1}{q_0 -1}}} < \delta.
\end{equation}
\end{lem}
\begin{proof}
By \eqref{eq-tuv}, \eqref{eq-a-75} and \eqref{eq-a-81}, we can choose a constant $C > 0$ independent of $\ep > 0$ such that
\[\|\tv_{\ep}\|_{L^{p+1}(\Omega_{\ep})} = \lambda_{\ep}^{\frac{n}{p+1} - \frac{2s(q_{\ep}+1)}{pq_{\ep}-1}} \|v_{\ep}\|_{L^{p+1}(\Omega)} \le \|v_{\ep}\|_{L^{p+1}(\Omega)} = \|u_{\ep}\|_{L^{q_{\ep}+1}(\Omega)}^{q_{\ep}+1 \over p+1} \le C.\]
Using this with \eqref{eq-tw-con}, we deduce that
\begin{equation}\label{eq-tuv-2}
\sup_{\ep > 0} \(\|\tu_{\ep}\|_{L^{q_{\ep}+1}(\Omega_{\ep})} + \|\tv_{\ep}\|_{L^{p+1}(\Omega_{\ep})}\) < \infty.
\end{equation}
Besides we have
\begin{equation}\label{eq-b-10}
\begin{aligned}
\sup_{\ep > 0} \|A_{\ep}\|_{L^{\frac{q_0+1}{q_0-1}}}^{\frac{q_0+1}{q_0-1}}
& = \sup_{\ep >0}\ \int_{B^n(0,r) \cap \kappa (\Omega_{\ep})} |x|^{-\theta(q_0)(q_0+1)} \whf_{\ep}^{q_0+1}(x,0) dx
\\
&\underset{\eqref{eq-bound-1}}{\le} \sup_{\ep >0}\int_{B^n(0,r) \cap \kappa (\Omega_{\ep})} |x|^{-\theta(q_0)(q_0+1)} F_{\ep}^{q_0+1}(x,0) dx + C r^{n-\theta(q_0)(q_0+1)}
\\
&\underset{\eqref{eq-bound-6}}{=} \sup_{\ep >0} \int_{B^n(0,1/r)^c \cap \Omega_{\ep}} \tu_{\ep}^{q_0+1} (x) dx + Cr^{n-\theta(q_0)(q_0+1)}.
\end{aligned}
\end{equation}
It holds that $q_0 > q_{\ep}$ and $\tu_{\ep} \le 1$ in $\Omega_{\ep}$.
Accordingly, from \eqref{eq-tuv-2} and the relation $\theta(q_0)(q_0+1) < n$, we get that $A_{\ep}$ is uniformly bounded in $L^{(q_0+1)/(q_0-1)}$.
In a similar way, we can check that $\{B_{\ep}\}_{\ep>0}$ is uniformly bounded in $L^{(p+1)/(p-1)}$, obtaining \eqref{eq-A-ep}.

On the other hand, \eqref{eq-A-ep-d} comes from Lemma \ref{lem-tutv-2} and \eqref{eq-b-10}. The proof is finished.
\end{proof}

We now recall the \textit{doubly weighted Hardy-Littlewood-Sobolev inequality} which was derived by Stein and Weiss \cite{SW}.
\begin{equation}\label{eq-dhls}
\left| \int_{\R^n} \int_{\R^n} \frac{f(x)g(y)}{|x|^{\alpha} |x-y|^{\lambda} |y|^{\beta}} dx dy \right| \le C_{\alpha, \beta, r_0, \lambda, n} \|f\|_{r_0} \|g\|_{r_1}
\end{equation}
where $1 < r_0, r_1 <\infty$, $0 < \lambda < n$, $\alpha + \beta \ge 0$ and
\[1-\frac{1}{r_0} - \frac{\lambda}{n} < \frac{\alpha}{n} < 1-\frac{1}{r_0}
\quad \text{and} \quad
\frac{1}{r_0} + \frac{1}{r_1} + \frac{\lambda + \alpha +\beta}{n} =2.\]
We are ready to give the proof of Proposition \ref{lem-gu}.
\begin{proof}[Proof of Proposition \ref{lem-gu}]
By utilizing the Hardy-Littlewood-Sobolev inequality \eqref{eq-hls}, its weighted version \eqref{eq-dhls}, H\"older's inequality and Lemma \ref{lem-a-2}, we find that for any sufficiently large $d_1 > 1$ and small $\ep > 0$,
\begin{equation}\label{eq-bound-4}
\begin{aligned}
\| K_{\ep} L_{\ep}\whg_{\ep}(\cdot,0) \|_{L^{d_1}} &= \left\| |x|^{-(n-2s)} * \( |x|^{-2\theta(q_0)} A_{\ep}(x) L_{\ep} \whg_{\ep}(x,0) \) \right\|_{L^{d_1}}\\
&\le C \left\| |\cdot|^{2\theta(p)} A_{\ep} L_{\ep} \whg_{\ep}(\cdot,0) \right\|_{L^{d_2}} \quad \text{($\theta(p) = - \theta(q_0)$ by \eqref{eq-theta})}
\\
&\le C \|A_{\ep}\|_{L^{\frac{q_0+1}{q_0-1}}} \left\| |\cdot|^{2\theta(p)} L_{\ep} \whg_{\ep}(\cdot,0) \right\|_{L^{d_3}}
\\
&\le C \|A_{\ep}\|_{L^{\frac{q_0+1}{q_0-1}}} \left\| |\cdot|^{2\theta(p)} \left[ |x|^{-(n-2s)} * \( |x|^{-2\theta(p)} B_{\ep}(x) \whg_{\ep}(x,0)\)\right] \right\|_{L^{d_3}}
\\
&\le C \|A_{\ep}\|_{L^{\frac{q_0+1}{q_0-1}}} \|B_{\ep} \whg_{\ep}(\cdot,0)\|_{L^{d_4}}
\\
&\le C \|A_{\ep}\|_{L^{\frac{q_0+1}{q_0-1}}} \|B_{\ep} \|_{L^{\frac{p+1}{p-1}}} \|\whg_{\ep}(\cdot,0)\|_{L^{d_5}}
:= C_{\ep} \|\whg_{\ep}(\cdot,0)\|_{L^{d_5}}
\end{aligned}
\end{equation}
where $d_2, d_3, d_4, d_5 > 1$ satisfy $2\theta(p) < n(1-1/d_4)$,
\[\frac{1}{d_2} - \frac{1}{d_1} = \frac{1}{d_4} - \frac{1}{d_3} = {2s \over n},
\quad \frac{1}{d_2} - \frac{1}{d_3} = \frac{q_0-1}{q_0+1}
\quad \text{and} \quad
\frac{1}{d_4} - \frac{1}{d_5} = \frac{p-1}{p+1},\]
which implies $d_5 = d_1$.
It is worth to point out that the condition $2s/(n-2s) < p < (n+2s)/(n-2s)$ allows us to find suitable parameters $\alpha, \beta, r_0, r_1$ to apply \eqref{eq-dhls}, provided that $d_1 > 1$ large enough.

By Lemma \ref{eq-AB}, it is possible to take $r > 0$ sufficiently small so that $\sup_{\ep > 0} C_{\ep}$ is as small as we want.
Hence \eqref{eq-bound-3} and \eqref{eq-bound-4} show
\[\| \whg_{\ep}(\cdot,0) \|_{L^d} \le C \( \| K_{\ep} L_{\ep} \whg_{\ep}(\cdot,0) \|_{L^d} + 1 \) \le \frac{1}{2} \| \whg_{\ep}(\cdot,0) \|_{L^d} + C\]
for any large $d > 1$.
Since it holds that $\whg_{\ep} \in L^{\infty}(B^n(0,r) \cap \kappa (\Omega_{\ep}))$ for each $\ep > 0$, the above inequalities imply that
\begin{equation}\label{eq-whg-l}
\| \whg_{\ep}(\cdot, 0) \|_{L^d} \le C \quad \text{for all } d \in [1,\infty).
\end{equation}

By making use of \eqref{eq-whg-l}, finally, we can estimate the blow-up rates of $\whf_{\ep}$ and $\whg_{\ep}$ near the origin, for which we divide the cases according to the value of $p$.
In the below, we assume that $x \in B^n(0,r) \cap \kappa (\Omega_{\ep})$.

\medskip \noindent \textsc{Case 1.} Suppose that $p \in (n/(n-2s), (n+2s)/(n-2s))$.
By \eqref{eq-bound-11} and \eqref{eq-bound-2}, we have
\begin{equation}\label{eq-bound-13}
\whf_{\ep}(x,0) \le |x|^{-(n-2s)} * \(|x|^{-(n+2s)+(n-2s)p} \whg_{\ep}^p(x,0)\) + C.
\end{equation}
Applying \eqref{eq-whg-l} with sufficiently large $d \ge 1$, we deduce that the $L^{\infty}$-norm of $\whf_{\ep}$ is uniformly bounded in $\ep$. Since we know
\begin{equation}\label{eq-bound-14}
\whg_{\ep}(x,0) \le |x|^{-(n-2s)} * \(|x|^{-(n+2s)+(n-2s)q_0} \whf_{\ep}^{q_0}(x,0)\) + C
\end{equation}
from \eqref{eq-bound-12} and \eqref{eq-bound-9}, the $L^{\infty}$-norm of $\whg_{\ep}$ is also uniformly bounded.

\medskip \noindent \textsc{Case 2.} Suppose that $p \in (2s/(n-2s), n/(n-2s))$. By \eqref{eq-whg-l} and \eqref{eq-bound-13}, it holds that
\[\whf_{\ep}(x,0) \le C |x|^{-n + (n-2s)p - \delta},\]
for any given $\delta > 0$. Using this and \eqref{eq-bound-14}, we discover
\[\whg_{\ep}(x,0) \le C |x|^{-(n-2s)} * \left[ |x|^{-(n+2s) + (n-2s)q_0} |x|^{(-n+(n-2s)p-\delta)q_0} \chi_B(x) \right] + C \le C.\]
We insert the above estimate into \eqref{eq-bound-13} again, getting
\[\whf_{\ep}(x,0) \le C |x|^{-n + (n-2s)p}.\]

\medskip \noindent \textsc{Case 3.} Suppose that $p = n/(n-2s)$. A similar argument to in Step 2 shows
\[\whf_{\ep}(x,0) \le - C \log |x| \quad \text{and} \quad \whg_{\ep}(x,0) \le C.\]
The proof is finished.
\end{proof}

We conclude this subsection by observing the limit behavior of $(\tu_{\ep}, \tv_{\ep})$ as $\ep \to 0$.
\begin{cor}\label{cor-tutv-1}
There exists a function $\mcv \in L^{p+1}(\R^n)$ such that $\tv_{\ep} \rightharpoonup \mcv$ weakly in $L^{p+1}(\R^n)$ up to a subsequence.
Furthermore, if $\mcu$ is the pointwise limit of $\{\tu_{\ep}\}_{\ep > 0}$ (see Lemma \ref{lem-tutv-2}), then $(\mcu, \mcv)$ is a solution of
\begin{equation}\label{eq-lim}
\begin{cases}
\mcu(x) = (-\Delta)^{-s} \mcv^p(x) = g_{n,s} \int_{\R^n} |x-y|^{-(n-2s)} \mcv^p(y) dy &\text{for } x \in \R^n,
\\
\mcv(x) = (-\Delta)^{-s} \mcu^{q_0}(x) = g_{n,s} \int_{\R^n} |x-y|^{-(n-2s)} \mcu^{q_0}(y) dy &\text{for } x \in \R^n,
\\
\mcu,\, \mcv > 0 &\text{in } \R^n.
\end{cases}
\end{equation}
\end{cor}
\begin{proof}
By \eqref{eq-tuv-2}, $\tv_{\ep} \to \mcv$ weakly in $L^{p+1}(\R^n)$ for some function $\mcv$.
In fact, $\tv_{\ep} \to \mcv$ in $C^{\alpha}(\R^n)$ for some $\alpha \in (0,2)$ thanks to elliptic regularity.
We apply \eqref{eq-b-20} and \eqref{eq-green-decom} to find
\begin{equation}\label{eq-b-2}
\tv_{\ep}(x) = \int_{\Omega_{\ep}} \frac{g_{n,s}}{|x-y|^{n-2s}} \tu_{\ep}^{q_{\ep}}(y) dy
- \lambda_{\ep}^{-(n-2s)} \int_{\Omega_{\ep}} H(\lambda_{\ep}^{-1}x + x_{\ep}, \lambda_{\ep}^{-1}y + x_{\ep})\, \tu_{\ep}^{q_{\ep}}(y) dy
\end{equation}
for any $x \in \Omega_{\ep}$ (cf. Lemma \ref{lem-a-82}).
Then the boundedness of $H$ and Proposition \ref{lem-gu} imply that
the second term of the right-hand side is bounded by $C_x \lambda_{\ep}^{-(n-2s)}$ for each $x \in \Omega$ and a constant $C_x$ depending on $x$.
Hence its limit $\mcv$ satisfies the second equation in \eqref{eq-lim} in view of Lemma \ref{lem-lam}.

Meanwhile, since $W = \mcu^{q_0}$ is a minimizer of the Hardy-Littlewood-Sobolev inequality \eqref{eq-hls}, we obtain
\[(-\Delta)^{-s} ((-\Delta)^{-s} W)^p = \mu W^{1 \over q_0} \quad \text{in } \R^n\]
for some $\mu \in \R$, reasoning as in the proof of Lemma \ref{lem-3-1}.
Moreover \eqref{eq-a-17} gives that $\mu = 1$. Thus the above equation reads as
\[(-\Delta)^{-s} \mcv^p = (-\Delta)^{-s} ((-\Delta)^{-s} \mcu^{q_0})^p = \mcu \quad \text{in } \R^n,\]
which is the first equation in \eqref{eq-lim}.

From Lemma \ref{lem-a-2}, we know that the nonnegative function $\mcu$ is nontrivial.
In light of the first and second equations in \eqref{eq-lim}, $\mcu$ and $\mcv$ should be positive. The lemma is proved.
\end{proof}
\begin{rmk}
Chen-Li-Ou \cite{CLO} showed that any solution to \eqref{eq-lim} is radially symmetric with respect to some point, say, the origin.
Also, the above corollary and Proposition \ref{lem-gu} are consistent with the result of Villavert \cite[Theorem 3]{V} and Chen-Li-Ou \cite{CLO2}
where the asymptotic behavior of solutions to \eqref{eq-lim} near $\infty$ is studied provided $1 < p < (n+2s)/(n-2s)$ and $p = 1$, respectively.
\end{rmk}

\section{Pointwise Limit of Rescaled Solutions away from the Singularity} \label{sec-Green}
Using the upper bound of the rescaled solutions $(\tu_{\ep}, \tv_{\ep})$ obtained in Proposition \ref{lem-gu},
we can now show the convergence of minimal energy solutions $(u_{\ep}, v_{\ep})$ to \eqref{eq-main}
to Green's function $G$ or the function $\wtg$ defined in \eqref{eq-tg-1} outside the concentration point $x_0 \in \Omega$.
Let $(\mcu, \mcv)$ be the limit of $(\tu_{\ep}, \tv_{\ep})$ that solves the system \eqref{eq-lim}.

\begin{lem}\label{lem-v}
Let $C_1 = \int_{\R^n} \mcu^{q_0} (x) dx$. Then we have
\[\lim_{\ep \to 0} \lambda_{\ep}^{\frac{n}{q_0+1}} v_{\ep}(x) = C_1 G(x,x_0) \quad \text{in } C^0 (\Omega \setminus \{x_0\}).\]
\end{lem}
\begin{proof}
According to Proposition \ref{lem-gu}, we have
\begin{equation}\label{eq-v-1}
\lim_{\ep \to 0} \lambda_{\ep}^{\frac{n}{q_0+1}} u_{\ep}^{q_\ep}(x) =0 \quad \text{in } C^0(\Omega \setminus \{x_0\}).
\end{equation}
On the other hand, employing Lemma \ref{lem-a-2}, Proposition \ref{lem-gu} and the dominated convergence theorem, we get
\begin{equation}\label{eq-v-2}
\lim_{\ep \to 0} \int_{\Omega} \lambda_{\ep}^{\frac{n}{q_0+1}} u_{\ep}^{q_{\ep}}(x) dx
= \lim_{\ep \to 0} \int_{\Omega_{\ep}} \tu^{q_{\ep}}(x) dx
= \int_{\R^n} \mcu^{q_0}(x) dx,
\end{equation}
Putting \eqref{eq-v-1} and \eqref{eq-v-2} together, we find
\[\lim_{\ep \to 0} \lambda_{\ep}^{\frac{n}{q_0+1}} u_{\ep}^{q_\ep}(x) = C_1 \delta_{x_0} (x).\]
Using this and \eqref{eq-2-29}, we deduce that
\[\lim_{\ep \to 0} \lambda_{\ep}^{\frac{n}{q_0+1}} v_{\ep}(x) = \lim_{\ep \to 0} \int_{\Omega} G(x,y) \lambda_{\ep}^{\frac{n}{q_0+1}} u_{\ep}^{q_{\ep}}(y) dy = C_1 G(x, x_0).\]
in $C^0(\Omega \setminus \{x_0\})$. The proof of the lemma is concluded.
\end{proof}

To treat the function $u_{\ep}$, we need to split the case according to the range of $p$,
because the decay estimate \eqref{eq-decay} alludes that $\mcv^p$ is not integrable in the entire space $\R^n$ if $p \le n/(n-2s)$.
\begin{lem}\label{lem-u}
The followings are true.
\begin{enumerate}
\item Assume that $p \in (n/(n-2s), (n+2s)/(n-2s))$ and let $C_2 = \int_{\R^n} \mcv^p(x) dx$. Then
\[\lim_{\ep \to 0} \lambda_{\ep}^{\frac{n}{p+1}} u_{\ep}(x) = C_2 G(x,x_0) \quad \text{in } C^0(\Omega \setminus \{x_0\}).\]
\item Assume that $p = n/(n-2s)$ and let $C_3 = (g_{n,s} C_1)^{\frac{n}{n-2s}}|\ms^{n-1}|$. Then
\[\lim_{\ep \to 0} \frac{\lambda_{\ep}^{\frac{n}{p+1}}}{\log \lambda_{\ep}} u_{\ep}(x) = C_3 G(x,x_0) \quad \text{in } C^0(\Omega \setminus \{x_0\}).\]
\item Assume that $p \ge 1$ and $p \in (2s/(n-2s), n/(n-2s))$. Then
\[\lim_{\ep \to 0} \lambda_{\ep}^{\frac{np}{q_0+1}} u_{\ep}(x) = C_1^p \wtg(x,x_0) \quad \text{in } C^0(\Omega \setminus \{x_0\}).\]
\end{enumerate}
\end{lem}
\begin{proof}
\textsc{Case 1.} Suppose that $p \in (n/(n-2s), (n+2s)/(n-2s))$. This case can be handled as in the proof of Lemma \ref{lem-v}.
By \eqref{eq-decay}, the function $\mcv^p$ is integrable in $\R^n$.
Thus Proposition \ref{lem-gu} gives
\[\lim_{\ep \to 0} \lambda_{\ep}^{\frac{n}{p+1}} v_{\ep}^{p_{\ep}}(x) = C_2 \delta_{x_0} (x).\]
Accordingly, we obtain from Green's representation formula \eqref{eq-2-29} that
\[\lim_{\ep \to 0} \lambda_{\ep}^{\frac{n}{p+1}} u_{\ep}(x) dx = \lim_{\ep \to 0} \int_{\Omega} G(x,y) \lambda_{\ep}^{\frac{n}{p+1}} v_{\ep}^p(y) dy = C_2 G(x,x_0)\]
in $C^0(\Omega \setminus \{x_0\})$.

\medskip \noindent \textsc{Case 2.} Assume that $p = n/(n-2s)$. By \eqref{eq-2-29}, it holds that
\[\lambda_{\ep}^{\frac{n}{p+1}} u_{\ep}(x)
= G(x,x_{\ep}) \int_{\Omega} \lambda_{\ep}^{\frac{n}{p+1}} v_{\ep}^p (y) dy
+ \int_{\Omega} [G(x,y) - G(x,x_{\ep})] \lambda_{\ep}^{\frac{n}{p+1}} v_{\ep}^p(y) dy.\]
Choose $r > 0$ so small that $\text{dist}(x, x_{\ep}) > 2r$. By virtue of Proposition \ref{lem-gu}, we have
\begin{align*}
\int_{B^n(x_{\ep}, r)} |G(x,y) - G(x,x_{\ep})| \lambda_{\ep}^{\frac{n}{p+1}} v_{\ep}^p(y) dy
&\le C \int_{B^n(x_{\ep}, r)} |y-x_{\ep}| \lambda_{\ep}^{\frac{n}{p+1}} v_{\ep}^p(y) dy
\\
&\le C \lambda_{\ep}^{-1} \int_{B^n(0, \lambda_{\ep} r)} |y| \tv^p(y) dy
\\
&\le C \lambda_{\ep}^{-1} \int_{B^n(0, \lambda_{\ep} r)} (1+|y|)^{-(n-1)} dy
\\
&\le C
\end{align*}
and
\[\int_{\Omega \setminus B^n(x_{\ep},r)} |G(x,y) - G(x,x_{\ep})| \lambda_{\ep}^{\frac{n}{p+1}} v_{\ep}^p(y) dy
\le C \lambda_{\ep}^{\frac{n}{p+1}} \cdot (\lambda_{\ep}^{n \over p+1} \lambda_{\ep}^{-(n-2s)})^p = C.\]
Therefore
\[\lim_{\ep \to 0} \frac{\lambda_{\ep}^{\frac{n}{p+1}}}{\log \lambda_{\ep}} u_{\ep}(x)
= \lim_{\ep \to 0} \( \frac{1}{\log \lambda_{\ep}} \int_{\Omega} \lambda_{\ep}^{\frac{n}{p+1}} v_{\ep}^{p} (y) dy \) G(x,x_0)\]
in $C^0(\Omega \setminus \{x_0\})$. Applying Lemma \ref{lem-b-2}, we obtain the desired result.

\medskip \noindent \textsc{Case 3.} Assume that $p \ge 1$ and $p \in (2s/(n-2s), n/(n-2s))$. From Proposition \ref{lem-gu}, we see
\[ (\lambda_{\ep}^{\frac{n}{q_0+1}} v_{\ep})^p (y)
\le{C \over |y-x_{\ep}|^{(n-2s)p}}\quad \forall ~y \in \Omega \setminus \{x_{\ep}\}.\]
Hence we deduce with the dominated convergence theorem, Lemma \ref{lem-v} and \eqref{eq-tg-1} that
\[\lim_{\ep \to 0} \lambda_{\ep}^{\frac{np}{q_0+1}} u_{\ep}(x) = \lim_{\ep \to 0} \int_{\Omega} G(x,y) (\lambda_{\ep}^{\frac{n}{q_0+1}} v_{\ep})^p (y) dy
= C_1^p \int_{\Omega} G(x,y) G^p(y,x_0) dy = C_1^p \wtg(x,x_0)\]
in $C^0(\Omega \setminus \{x_0\})$. The proof is completed.
\end{proof}
\begin{proof}[Proof of Theorem \ref{thm-1}] It follows directly from Lemmas \ref{lem-bb}, \ref{lem-v} and \ref{lem-u}.
\end{proof}

\appendix
\section{The Miscellaneous}\label{app-a}
This section is devoted to prove two technical lemmas needed in the proof of the main theorems.

\medskip
The first lemma describes the continuity and compactness property of the inverse fractional Laplacian $(-\Delta)^{-s}$ in $\Omega$.
\begin{lem}\label{lem-cpt}
For a smooth bounded domain $\Omega \subset \R^n$, let $G$ be Green's function of the fractional Dirichlet Laplacian $(-\Delta)^s$ in $\Omega$ defined in Subsection \ref{subsec_green}
and $(-\Delta)^{-s}$ the inverse fractional Laplacian given by
\[(-\Delta)^{-s} f(x) = \int_{\Omega} G(x,y) f(y) dy \quad \text{for } x \in \Omega\]
for $f \in L^{r_0}(\Omega)$ with any fixed $r_0 \in (1, \infty)$.
Suppose that $r_1 \in (1,\infty)$ is the number satisfying $1/r_0 - 1/r_1 = 2s/n$.
Then $(-\Delta)^{-s}: L^{r_0}(\Omega) \to L^r(\Omega)$ is bounded for any $r \in [1, r_1]$ and compact for $r \in [1, r_1)$.
\end{lem}
\begin{proof}
By \eqref{eq-G} and \eqref{eq-hls}, $\mcl_0 := (-\Delta)^{-s}$ is a bounded operator from $L^{r_0}(\Omega)$ to $L^{r_1}(\Omega)$.
We prove the compactness of $\mcl_0: L^{r_0}(\Omega) \to L^r(\Omega)$ for $r \in [1, r_1)$ in two steps.

\medskip \noindent \textbf{Step 1.}
Let us define an operator $\mcl_1: L^{r_0}(\Omega) \to L^r(\Omega)$  for any $r \in [1, r_1]$ by
\[\mcl_1 f(x) = \int_{\Omega} \frac{g_{n,s}}{|x-y|^{n-2s}} f(y) dy \quad \text{for } x \in \R^n,\]
which is continuous. We claim that it is compact for $r \in [1, r_1)$.
To justify this, it suffices to show that for any bounded sequence $\{f_k\}_{k \in \N}$ in $L^{r_0}(\Omega)$,
$\{\mcl_1 f_k\}_{k \in \N}$ possesses a convergent subsequence in $L^r(\Omega)$ whenever $r \in (r_0, r_1)$.

Putting $f_k = 0$ in $\R^n \setminus \Omega$, we set
\[F_k = \mathcal{L}_1 f_k \quad \text{and} \quad F_{k,\vep} = \int_{\{|\cdot-y| \ge \vep\}} \frac{g_{n,s}}{|\cdot-y|^{n-2s}} f_k(y) dy \quad \text{in } \R^n\]
for each $k \in \N$ and $\vep > 0$ small.
If $\theta \in (0,1)$ is chosen to be a number satisfying $1/r = \theta/r_0 + (1-\theta)/r_1$,
then it follows from Young's inequality and the Hardy-Littlewood-Sobolev inequality \eqref{eq-hls} that
\begin{align*}
\|F_k - F_{k,\vep}\|_{L^r(\R^n)} &\le \|F_k - F_{k,\vep}\|_{L^{r_0} (\R^n)}^{\theta} \|F_k - F_{k,\vep}\|_{L^{r_1}(\R^n)}^{1-\theta} \\
&\le C \(\vep^{2s} \|f_k\|_{L^{r_0} (\R^n)}\)^{\theta} \cdot \|f_k\|_{L^{r_0} (\R^n)}^{1-\theta} \le C \vep^{2s\theta}
\end{align*}
where $C > 0$ is independent of $k \in \N$.
For each fixed $\vep > 0$, $\{F_{k,\vep}\}_{k \in \N}$ is precompact in $L^r(\Omega)$
because their kernels are uniformly bounded in $k \in \N$ and the domain $\Omega$ is bounded.
Therefore $F_{k,\vep}$ converges to a function $F_{0,\vep}$ in $L^r(\Omega)$ passing to a subsequence, and
\begin{equation}\label{eq-e-2}
\|F_{0,\vep_1} - F_{0,\vep_2}\|_{L^r(\Omega)} = \lim_{k \to \infty} \|F_{k,\vep_1} - F_{k,\vep_2}\|_{L^r(\Omega)} \le C\(\vep_1^{2s\theta} + \vep_2^{2s\theta}\).
\end{equation}
As a result, for any sequence $\{\vep_m\}_{m \in \N} \subset (0, \infty)$ which tends to 0,
$\{F_{0,\vep_m}\}_{m \in \N}$ subconverges to $F_0 \in L^r(\Omega)$.
Now one may apply Cantor's diagonal argument to deduce $F_k \to F_0$ in $L^r(\Omega)$ along a subsequence.
The assertion is proved.

\medskip \noindent \textbf{Step 2.} Let us complete the proof.
Fix a value $r \in [1, r_1)$ and a bounded sequence $\{f_k\}_{k \in \N}$ in $L^{r_0}(\Omega)$.
Since $\{\mcl_0 f_k\}_{k \in \N}$ is bounded in $L^{r_1}(\Omega)$, we can find a function $g \in L^{r_1}(\Omega)$ such that $\mcl_0 f_k \rightharpoonup g$ in $L^{r_1}(\Omega)$ up to a subsequence.
The main task is to prove that $\mcl_0 f_k \to g$ in $L^r(\Omega)$ after passing to a subsequence.
A subtle issue arises due to the singular behavior of $H(x,y)$ for the variable $x$ near the boundary $\pa \Omega$.
To avoid this technicality, we again proceed a diagonalization argument.

Because of \cite[Lemma 2.4]{CKL}, for each $\delta > 0$, there is a constant $C_{\delta} > 0$ such that
\[|H(x,y)| + |\nabla_x H(x,y)| \le C_{\delta} \quad \text{for all } x \in \mci(\Omega,\delta),\, y \in \Omega\]
where $\mci(\Omega,\delta) = \{ x \in \Omega : \text{dist}(x, \pa \Omega) > \delta\}$.
By this property, the operator
\[\mcl_2f(x) := \int_{\Omega} H(x,y) f(y) dy \quad \text{for } x \in \Omega \]
is compact from $L^{r_0}(\Omega)$ to $L^r(\mci(\Omega,\delta))$ for arbitrary $r \in [1, \infty)$.
In view of the assertion proved in the previous step, $\mcl_0: L^{r_0}(\Omega) \to L^r(\mci(\Omega,\delta))$ is compact for any $r \in [1, r_1)$.

Pick any sequence $\{\delta_k\}_{k \in \N}$ of small numbers converging to 0.
Moreover we construct functions $f_{km}$ for $k, m \in \N$ as follows: For each $k \in \N$,
\begin{itemize}
\item[-] $\{f_{(k+1)m}\}_{m \in \N}$ is a subsequence of  $\{f_{km}\}_{m \in \N}$;
\item[-] $\mcl_0 f_{km} \to g$ in $L^r(\mci(\Omega,\delta_k))$ as $m \to \infty$.
\end{itemize}
Then applying H\"older's inequality gives
\begin{equation}\label{eq-a-2}
\begin{aligned}
\|\mcl_0 f_{kk} - g\|_{L^r(\Omega)} &\le \|\mcl_0 f_{kk} - g\|_{L^r(\mci(\Omega,\delta_m))} + \|\mcl_0 f_{kk} - g\|_{L^r(\mco(\Omega,\delta_m))}
\\
&\le \|\mcl_0 f_{kk} - g\|_{L^r(\mci(\Omega,\delta_m))} + C \delta_m^{\frac{1}{r} - \frac{1}{r_1}} \( \sup_{k \in \N} \|\mcl_0 f_{kk}\|_{L^{r_1}(\Omega)} + \|g\|_{L^{r_1}(\Omega)} \)
\end{aligned}
\end{equation}
for arbitrary $m \in \N$ where $\mco(\Omega,\delta) = \{ x \in \Omega : \text{dist}(x, \pa \Omega) < \delta\}$.
Since $\mcl_0 f_{kk} \to g$ in $L^r(\mci(\Omega,\delta_m))$ as $k \to \infty$, we can deduce from \eqref{eq-a-2} that
\[\limsup_{k \to \infty} \|\mcl_0 f_{kk} - g\|_{L^r(\Omega)} \le C \delta_m^{\frac{1}{r} - \frac{1}{r_1}} \quad \text{for all } m \in \N.\]
Hence $\mcl_0 f_{kk} \to g$ in $L^r(\Omega)$ as $k \to \infty$ whenever $r \in [1, r_1)$. This completes the proof.
\end{proof}

In the remaining part of this appendix, we shall prove Lemma \ref{lem-a-4}.
Recalling the notations $\lambda_{\ep}$, $x_{\ep}$ and $\Omega_{\ep}$ given in Section \ref{sec-4}, we set $(-\Delta)_{\ep}^{-s}: L^{\infty}(\Omega_{\ep}) \to L^{\infty}(\Omega_{\ep})$ by
\[(-\Delta)_{\ep}^{-s} f(x) = \int_{\Omega_{\ep}} G_{\ep}(x,y) f(y) dy \quad \text{for all } x \in \Omega_{\ep}\]
for each $\ep > 0$ small, where
\begin{equation}\label{eq-G_ep}
G_{\ep}(x,y) := \lambda_{\ep}^{-(n-2s)} G(\lambda_{\ep}^{-1} x + x_{\ep},\, \lambda_{\ep}^{-1} y + x_{\ep}) \quad \text{for every } (x,y) \in \Omega_{\ep} \times \Omega_{\ep}.
\end{equation}
\begin{lem}\label{lem-a-82}
The function $\tw_{\ep}$ defined in \eqref{eq-tw} satisfies
\[
\tw_{\ep}^{\frac{1}{q_{\ep}}} = (-\Delta)_{\ep}^{-s} ((-\Delta)_{\ep}^{-s} \tw_{\ep})^p \quad \text{in } \Omega_{\ep}.\]
\end{lem}
\begin{proof}
By \eqref{eq-tw} and the definition of $G_{\ep}$, it holds that
\[(-\Delta)_{\ep}^{-s} \tw_{\ep}(x) = \lambda_{\ep}^{2s - \alpha_{\ep} q_{\ep}} ((-\Delta)^{-s} w_{\ep})(\lambda_{\ep}^{-1}x + x_{\ep}) \quad \text{for } x \in \Omega_{\ep}.\]
Making use of \eqref{eq-tw}, \eqref{eq-w-1}, \eqref{eq-ab}, the above equality and a change of variables, we compute
\begin{align*}
\tw_{\ep}^{\frac{1}{q_{\ep}}}(x) &= \lambda_{\ep}^{-\alpha_{\ep}} (-\Delta)^{-s} ((-\Delta)^{-s} w_{\ep})^p(\lambda_{\ep}^{-1}x + x_{\ep}) \\
&= \lambda_{\ep}^{-\alpha_{\ep}} \int_{\Omega} G(\lambda_{\ep}^{-1} x + x_{\ep}, \, y)((-\Delta)^{-s} w_{\ep})^p(y) dy \\
&= \lambda_{\ep}^{-\alpha_{\ep}} \int_{\Omega_{\ep}} \lambda_{\ep}^{n-2s} G_{\ep}(x,y) \cdot \lambda_{\ep}^{(\alpha_{\ep} q_{\ep} - 2s)p} ((-\Delta)_{\ep}^{-s} \tw_{\ep})^p(y) \cdot \lambda_{\ep}^{-n} dy \\
&= (-\Delta)_{\ep}^{-s} ((-\Delta)_{\ep}^{-s} \tw_{\ep})^p (x) \end{align*}
for $x \in \Omega_{\ep}$. The lemma is proved.
\end{proof}

\begin{proof}[Proof of Lemma \ref{lem-a-4}]
In this proof, we use a convention that the inverse fractional Laplacian $(-\Delta)^{-s}$ is the operator defined in \eqref{eq-rep-2}.

\medskip
Since $(q_{\ep}+1) \alpha_{\ep} -n > 0$, we deduce from \eqref{eq-a-34} and \eqref{eq-a-81} that
\[\|\tw_{\ep}\|_{L^{ \frac{q_{\ep}+1}{q_{\ep}}} (\Omega_{\ep})} \le C\]
for some constant $C > 0$ independent of $\ep > 0$. Recall also that $\|\tw_{\ep}\|_{L^{\infty}(\Omega_{\ep})} \le C$.
Thus, for an arbitrary number $\zeta_1 > (q_0+1)/q_0$ and small $\ep > 0$, it holds that
\begin{equation}\label{eq-a-51}
\|\tw_{\ep}\|_{L^{\zeta_1} (\Omega_{\ep})} \le C
\end{equation}
and $\tw_{\ep}$ converges weakly to a certain function $W$ in $L^{\zeta_1}(D)$ along a subsequence.
By Eq. \eqref{eq-b-20} and elliptic regularity (see \cite{CDDS, CaS, CSt}), the family $\{(\tu_{\ep}, \tv_{\ep})\}_{\ep > 0}$ is uniformly bounded in $(C^{\alpha}(D))^2$ for some $\alpha \in (0,2)$.
Therefore the same is true for $\{\tw_{\ep}\}_{\ep > 0}$.
Moreover, if we fix any small $\zeta_1 > (q_0+1)/q_0$ and choose $\zeta_2 \in (1, \infty)$ such that $1/\zeta_1 + 2s/n = 1/\zeta_2$,
then we get from \eqref{eq-hls} and \eqref{eq-a-51} that
\[\|(-\Delta)_{\ep}^{-s} \tw_{\ep}\|_{L^{\zeta_2}(D)} \le C \|\tw_{\ep}\|_{L^{\zeta_1}(\Omega_{\ep})} \le C.\]
Hence $(-\Delta)_{\ep}^{-s} \tw_{\ep}$ converges weakly to some $V$ in $L^{\zeta_2}(D)$.

For a nonnegative function $\phi \in C_c^{\infty}(D)$, if $\ep > 0$ is so small that $\text{supp}\, \phi \subset D$, the symmetry property of $G_{\ep}$ guarantees that
\begin{equation}\label{eq-a-40}
\int_{\Omega_{\ep}} (-\Delta)_{\ep}^{-s} \tw_{\ep} (x) \phi (x) dx = \int_{\Omega_{\ep}} \tw_{\ep} (-\Delta)_{\ep}^{-s} \phi (x) dx.
\end{equation}
Note that for any nonnegative $\phi \in C_c^{\infty}(D)$,
\begin{equation}\label{eq-a-41}
\begin{aligned}
\limsup_{\ep \to 0}\( (-\Delta)_{\ep}^{-s} \phi (x)\)
&\le \int_{\R^n}\left[\limsup_{\ep \to 0}\lambda_{\ep}^{-(n-2s)}  G(\lambda_{\ep}^{-1} x +x_{\ep},\, \lambda_{\ep}^{-1}y +x_{\ep})\right] \phi (y) dy
\\
&\le  \int_{\R^n}\frac{g_{n,s}}{|x-y|^{n-2s}} \phi (y) dy = (-\Delta)^{-s} \phi (x).
\end{aligned}
\end{equation}
Taking \eqref{eq-a-41} into account, we send $\ep \to 0$ in \eqref{eq-a-40} to get
\[\int_{D} V(x) \phi (x) dx \le \int_{D} W (x) (-\Delta)^{-s} \phi (x) dx = \int_{D} (-\Delta)^{-s} W (x) \phi (x) dx,\]
which shows that $V(x) \le (-\Delta)^{-s} W(x)$ for a.e. $x \in D$.
As a consequence, taking $\ep \to 0$ in the identity
\[\int_{\Omega_{\ep}} \tw_{\ep}^{\frac{1}{q_{\ep}}}(x) \phi (x) dx = \int_{\Omega_{\ep}} (-\Delta_{\ep})^{-s} \phi (x) ((-\Delta_{\ep})^{-s} \tw_{\ep})^{p} (x) dx\]
which was verified in Lemma \ref{lem-a-82}, we discover
\[\int_{D} W^{1 \over q_0}(x) \phi(x) dx \le \int_{D} (-\Delta)^{-s} \phi (x) V(x)^p dx \le \int_{D} (-\Delta)^{-s} \phi (x) ((-\Delta)^{-s} W)^p (x) dx.\]
As a result, \eqref{eq-a-13} is proved. Estimate \eqref{eq-W} is a simple consequence of \eqref{eq-a-51} and \eqref{eq-hls}.

\medskip
We are only left to prove that $W$ is nontrivial.
Applying \cite[Lemma 2.11]{CDDS} or \cite[Theorem 1.3]{CSt} into Eq. \eqref{eq-b-20},
we obtain that there exists $C > 0$ independent of $\ep > 0$ such that
\[\tu_{\ep} (x) \le C \textrm{dist}(x, \pa \Omega_{\ep})^{\min\{2s, 1\}} \quad \text{for all } x \in \Omega_{\ep}.\]
Hence putting $x = 0$ yields that
\[\text{dist}(0, \pa \Omega_{\ep}) \ge 2c_0 \quad \text{for some } c_0 > 0.\]
Since the family $\{\tw_{\ep}\}_{\ep > 0}$ is bounded in $C^{\alpha}(B^n(0,c_0))$ for some $\alpha \in (0,2)$ and $\tw_{\ep}(0) = 1$ for each $\ep > 0$,
we conclude that the limit $W$ of $\tw_{\ep}$ is nonzero. The proof is finished.
\end{proof}

\section{Further estimate on $\tv_{\ep}$}\label{app-b}
In this appendix, we give estimates on the sharp decay and the $L^p(\Omega_{\ep})$-norm of the scaled solution $\tv_{\ep}$, which was defined in \eqref{eq-tuv}.
This part is crucial in proving Lemma \ref{lem-u} (2).
\begin{lem}\label{lem-b-1}
Given any $\delta > 0$, there exist a small number $r > 0$ and large $R > 0$ such that
\begin{equation}\label{eq-b-11}
(1-\delta) g_{n,s} C_1 |x|^{-(n-2s)} \le \tv_{\ep}(x) \le (1+\delta) g_{n,s} C_1 |x|^{-(n-2s)}
\end{equation}
holds for any $x \in \Omega_{\ep}$ with $R \le |x| \le \lambda_{\ep} r$. Here $C_1 = \int_{\R^n} \mcu^{q_0}(x) dx$.
\end{lem}
\begin{proof}
We will calculate each of the terms in the right-hand side of \eqref{eq-b-2}.

It holds that
\[\sup_{|x| \le \lambda_{\ep}} \sup_{y \in \Omega_{\ep}} H(\lambda_{\ep}^{-1} x + x_{\ep}, \lambda_{\ep}^{-1} y + x_{\ep}) \le C < \infty.\]
Thus, for $r \in (0,1)$ small and $|x| \le \lambda_{\ep} r$, the second term in the right-hand side of \eqref{eq-b-2} is computed as
\begin{equation}\label{eq-b-5}
\begin{aligned}
\int_{\Omega_{\ep}} \lambda_{\ep}^{-(n-2s)} H(\lambda_{\ep}^{-1} x + x_{\ep}, \lambda_{\ep}^{-1} y + x_{\ep}) \tu_{\ep}^{q_{\ep}}(y) dy
&\le C\lambda_{\ep}^{-(n-2s)} \int_{\Omega_{\ep}} \tu_{\ep}^{q_{\ep}}(y) dy \\
&\le C \lambda_{\ep}^{-(n-2s)} \le {\delta \over 3} |x|^{-(n-2s)}.
\end{aligned}
\end{equation}
To estimate the first term in the right-hand side of \eqref{eq-b-2}, we split it by
\begin{equation}\label{eq-b-3}
\int_{\Omega_{\ep}}\frac{g_{n,s}}{|x-y|^{n-2s}} \tu_{\ep}^{q_{\ep}}(y) dy
= \int_{B^n(0, r|x|)} \cdots + \int_{\Omega_{\ep}\setminus B^n(0, r|x|)} \cdots.
\end{equation}
Employing Proposition \ref{lem-gu}, we obtain
\begin{equation}\label{eq-b-4}
\begin{aligned}
\int_{\Omega_{\ep}\setminus B^n(0, r|x|)} \frac{g_{n,s}}{|x-y|^{n-2s}} \tu_{\ep}^{q_{\ep}}(y) dy
&\le C \int_{\Omega_{\ep} \setminus B^n(0, r|x|)} \frac{1}{|x-y|^{n-2s}} \(\frac{\log |y|}{|y|^{n-2s}}\)^{q_{\ep}} dy \\
&\le \frac{C}{r^m} |x|^{- \frac{n^2}{n-2s}} \le {\delta \over 3} |x|^{-(n-2s)}
\end{aligned}
\end{equation}
for all $|x| > R$ and some $m > 0$, where the last inequality holds provided that $R$ is sufficiently large compared to $r^{-1}$.
Furthermore, since $(1-r)|x| \le |x-y| \le (1+r)|x|$ for $y \in B^n(0,r|x|)$ and
\[\lim_{\ep \to 0} \int_{\R^n} \tu_{\ep}^{q_{\ep}}(y) dy = \int_{\R^n} \mcu^{q_0}(y) dy = C_1,\]
one can choose $R > 1$ large and $r > 0$ small so that
\[\(1-{\delta \over 3}\) \frac{g_{n,s}C_1}{|x|^{n-2s}} \le \int_{B^n(0, r|x|)}\frac{g_{n,s}}{|x-y|^{n-2s}} \tu_{\ep}^{q_{\ep}}(y) dy
\le \(1+{\delta \over 3}\) \frac{g_{n,s}C_1}{|x|^{n-2s}}\]
for every $|x| > R$.

Putting this, \eqref{eq-b-5}-\eqref{eq-b-4} and \eqref{eq-b-2} together, we finally get the desired estimate \eqref{eq-b-11}.
\end{proof}

\begin{lem}\label{lem-b-2}
Let $p = n/(n-2s)$. Then
\[\lim_{\ep \to 0} \frac{1}{\log \lambda_{\ep}} \int_{\Omega_{\ep}} \tv_{\ep}^{p} (x) dx = (g_{n,s} C_1)^{\frac{n}{n-2s}} |\ms^{n-1}|.\]
\end{lem}
\begin{proof}
We write
\[\int_{\Omega_{\ep}} \tv_{\ep}^p(x) dx
= \int_{\{R \le |x| \le \lambda_{\ep} r\}} \tv_{\ep}^p(x) dx + \int_{B^n(0,R)} \tv_{\ep}^p (x) dx + \int_{\Omega_{\ep} \setminus B^n(0, \lambda_{\ep} r)} \tv_{\ep}^p(x) dx\]
compute each term in the right-hand side.
For the first term, we use Lemma  \ref{lem-b-1} to find that
\begin{align*}
(1-\delta)^{\frac{n}{n-2s}} (g_{n,s} C_1)^{\frac{n}{n-2s}} |\ms^{n-1}|  \log \lambda_{\ep} - O(1)
&\le \int_{\{R \le |x| \le \lambda_{\ep} r\}} \tv_{\ep}^p(x) dx \\
&\le (1+\delta)^{\frac{n}{n-2s}} (g_{n,s} C_1)^{\frac{n}{n-2s}} |\ms^{n-1}| \log \lambda_{\ep} + O(1).
\end{align*}
On the other hand, Proposition \ref{lem-gu} shows
\[\int_{B^n(0,R)} \tv_{\ep}^p (x) dx + \int_{\Omega_{\ep} \setminus B^n(0, \lambda_{\ep} r)} \tv_{\ep}^p(x) dx = O(1).\]
Combining the above estimates yields
\begin{align*}
(1-\delta)^{\frac{n}{n-2s}} (g_{n,s} C_1)^{\frac{n}{n-2s}} |\ms^{n-1}|
&\le \liminf_{\ep \to 0} \frac{1}{\log \lambda_{\ep}} \int_{\Omega_{\ep}} \tv_{\ep}^{p} (x) dx
\\
&\le \limsup_{\ep \to 0} \frac{1}{\log \lambda_{\ep}} \int_{\Omega_{\ep}} \tv_{\ep}^{p} (x) dx
\le (1+\delta)^{\frac{n}{n-2s}} (g_{n,s} C_1)^{\frac{n}{n-2s}} |\ms^{n-1}|.
\end{align*}
Taking $\delta \to 0$ gives the result.
\end{proof}

\end{document}